\documentclass[11pt]{extarticle}
\usepackage[a4paper,bindingoffset=0.in,left=2.5cm,right=2.5cm,top=2cm,bottom=3cm]{geometry}

\usepackage{graphicx}
\usepackage{mathtools,amsmath,amssymb,amsthm}
\usepackage{mathrsfs}
\usepackage{comment,xcolor}

\usepackage{multirow}
\usepackage{makecell}
\usepackage{array,booktabs}
\usepackage{color, colortbl}

\usepackage{tabstackengine}[2016--11--30]
\usepackage{xcolor}
\usepackage[labelfont=bf]{caption}

\usepackage{algpseudocode}
\usepackage{siunitx}
\usepackage[affil-it]{authblk}
\usepackage[ruled,vlined,linesnumbered]{algorithm2e}
\usepackage{tikz}
\usetikzlibrary{decorations.pathreplacing,decorations.markings,arrows,arrows.meta,fit}
\tikzset{
  midarrow/.style={
    postaction={
      decorate,
      decoration={
        markings,
        mark=at position 0.45 with {\arrow{>}}
      }
    }
  }
}
\usepackage{float}
\usepackage{enumitem}

\usepackage[
    backend=biber,
    style=numeric,
    maxbibnames=99,
            doi=false,
    url=false,
    isbn=false,
    eprint=false]{biblatex}

\usepackage{aliascnt}

\usepackage{aliascnt}

\newtheorem{theorem}{Theorem}[section]

\newaliascnt{proposition}{theorem}
\newtheorem{proposition}[proposition]{Proposition}
\aliascntresetthe{proposition}

\newaliascnt{property}{theorem}

\aliascntresetthe{property}

\newaliascnt{lemma}{theorem}
\newtheorem{lemma}[lemma]{Lemma}
\aliascntresetthe{lemma}

\newaliascnt{claim}{theorem}

\aliascntresetthe{claim}

\newaliascnt{assumptions}{theorem}

\aliascntresetthe{assumptions}

\newaliascnt{corollary}{theorem}
\newtheorem{corollary}[corollary]{Corollary}
\aliascntresetthe{corollary}

\newaliascnt{conjecture}{theorem}

\aliascntresetthe{conjecture}

\newaliascnt{problem}{theorem}
\newtheorem{problem}[problem]{Problem}
\aliascntresetthe{problem}

\theoremstyle{definition}

\newaliascnt{definition}{theorem}
\newtheorem{definition}[definition]{Definition}
\aliascntresetthe{definition}

\newaliascnt{remark}{theorem}
\newtheorem{remark}[remark]{Remark}
\aliascntresetthe{remark}

\newaliascnt{example}{theorem}
\newtheorem{example}[example]{Example}
\aliascntresetthe{example}

\renewenvironment{proof}{{\noindent\bfseries Proof.}}{\qed}

\usepackage[colorlinks=true,allcolors=blue!75!black]{hyperref}
\usepackage[capitalise,noabbrev]{cleveref}

\crefname{theorem}{Theorem}{Theorems}
\crefname{proposition}{Proposition}{Propositions}
\crefname{property}{Property}{Properties}
\crefname{lemma}{Lemma}{Lemmas}
\crefname{claim}{Claim}{Claims}
\crefname{assumptions}{Assumption}{Assumptions}
\crefname{corollary}{Corollary}{Corollaries}
\crefname{conjecture}{Conjecture}{Conjectures}
\crefname{problem}{Problem}{Problems}
\crefname{definition}{Definition}{Definitions}
\crefname{remark}{Remark}{Remarks}
\crefname{example}{Example}{Examples}

\definecolor{ThistleBlue}{rgb}{102, 153, 255}

\newcommand{\diag}{\operatorname{diag}}

\newcommand{\sign}{\operatorname{sign}}
\newcommand{\Gr}{\operatorname{Gr}}

\newcommand{\R}{\mathbb{R}}
\renewcommand{\P}{\mathbb{P}}

\newcommand{\ba}{\mathbf{a}}
\newcommand{\bx}{\mathbf{x}}
\newcommand{\bB}{\mathbf{B}}

\newcommand{\bI}{\mathbf{I}}
\newcommand{\bS}{\mathbf{S}}
\newcommand{\bP}{\mathbf{P}}
\newcommand{\bC}{\mathbf{C}}
\newcommand{\bA}{\mathbf{A}}
\newcommand{\bM}{\mathbf{M}}

\newcommand{\xstar}{\mathbf{x}_{\star}}

\newcommand{\JacAB}{J_{\ba,\bB}}   
\newcommand{\Jx}{J_{\ba,\bB}({\xstar})}   

\DeclareMathAlphabet{\mathcal}{OMS}{cmsy}{m}{n}
\DeclareMathAlphabet{\mathcalbf}{OMS}{cmsy}{b}{n}
\DeclareMathAlphabet{\mathbfsfit}{\encodingdefault}{\sfdefault}{b}{it}
\DeclareMathAlphabet{\mathbfsf}{\encodingdefault}{\sfdefault}{b}{n}

\renewcommand{\leq}{\leqslant}
\renewcommand{\geq}{\geqslant}

\title{{\bf Stable Coexistence in Ecologies and Games}}
\author{{T\"urk\"u \"Ozl\"um Çelik\thanks{Max Planck Institute of Molecular Cell Biology and Genetics (MPI-CBG)
  and Center for Systems Biology Dresden (CSBD). celik@mpi-cbg.de \& zucal@mpi-cbg.de} , Vincenzo Antonio Isoldi\thanks{Max Planck Institute for Mathematics in the Sciences (MPI MiS). isoldi@mis.mpg.de \& mail@irem-portakal.de} , Irem Portakal\textsuperscript{$\dagger$}, Giulio Zucal\textsuperscript{$\ast$}}}
\date{}
\begin{document}

\maketitle

\begin{abstract}
We study feasible stable equilibria of Lotka--Volterra systems and their higher-order extensions. We complete the classification of impossible ecological interaction networks with at most four species and extend several of these impossibility results to families with arbitrarily many species. We then show that these sign-pattern obstructions are specific to the pairwise Lotka--Volterra model: arbitrary prescribed growth rates and pairwise coefficients can be supplemented by higher-order interactions so as to admit a feasible asymptotically stable equilibrium. Through the correspondence with replicator dynamics, we interpret feasible equilibria of higher-order Lotka--Volterra systems as totally mixed symmetric Nash equilibria of symmetric multiplayer games, derive bounds on their number, and study their robustness under perturbations of the payoff tensors. We conclude by showing that every impossible ecology determines a nonempty open class of symmetric two-player games with no totally mixed evolutionarily stable strategy.
\end{abstract}

\section{Introduction}

Stable coexistence is a central phenomenon in both ecological dynamics and evolutionary game theory \cite{Hofbauer1998}. In ecology, coexistence is modeled by an equilibrium at which all species have positive abundance, while stability asks whether such a state persists under small perturbations of the population densities. In evolutionary game theory, closely related objects are internal equilibria of replicator dynamics and totally mixed Nash equilibria, where every strategy is used with positive probability. This paper studies stable coexistence across these two settings, starting from ecological interaction networks and leading to higher-order games.

We begin with the generalized Lotka--Volterra system $\dot{\bx}=\diag(\bx)(\ba-\bB\bx),$ an ordinary differential equation that models the dynamics of interacting populations. Here $\bx(t)\in\mathbb R^n$ records species abundances, $\ba\in\mathbb R^n$ their intrinsic growth rates, and $\bB\in\mathbb R^{n\times n}$ their pairwise interactions. The signs of the entries of $(\ba,\bB)$ determine an \emph{ecological interaction network}. Once such a sign pattern is fixed, one may ask whether some choice of parameter magnitudes produces a stable equilibrium at which all species coexist with positive abundance. Following \cite{Haas2025}, a sign pattern for which this is impossible is called an \emph{impossible ecology}. 

The systematic study of impossible ecologies was initiated in \cite{Haas2025}, where the problem was classified for small systems and a list of candidate four-species obstructions was identified through numerical exploration. In \cite{Celik2025}, some candidates were proved impossible using combinatorial and computational tools through the lenses of nonlinear algebra. More broadly, classical sufficient conditions for the global stability of equilibria in generalized Lotka--Volterra systems were developed in \cite{goh1977global,takeuchi1978global}. In a complementary random-matrix approach, May established a complexity--stability threshold for randomly assembled community matrices \cite{may1972will}. Subsequent work refined this framework through interaction-type-dependent stability criteria and geometric analyses of feasibility \cite{allesina2012stability,grilli2017feasibility}. More recently, \cite{aspenberg2026stableequilibrialotkavolterraequations} obtained necessary conditions for the stability of Lotka--Volterra equilibria, complementing the classical sufficient conditions above. These studies primarily concern specified interaction matrices or statistical ensembles. By contrast, the impossible-ecology problem fixes only the signs of the growth rates and interaction coefficients and asks whether any choice of their magnitudes can produce a feasible, locally asymptotically stable equilibrium.

Our first contribution is to complete the four-species classification started in \cite{Haas2025, Celik2025}. Using tools from matrix analysis and spectral theory, together with the equilibrium equations and the sign constraints imposed by the ecological networks, in Section~\ref{sec: 3.3}, we prove that all eleven four-species candidates identified in \cite{Haas2025} are impossible. Several of the arguments extend beyond four species and yield infinite families of impossible ecologies as shown in Section~\ref{sec:impossible-ecologies}, including families built from obligate prey-predator cycles and competing mutualistic groups. Thus, the conjectural four-species classification is fully resolved, while some of its obstructions are shown to persist for arbitrarily many species.

Our second contribution is to show that these sign-pattern obstructions are specific to the pairwise Lotka--Volterra model. Once higher-order interactions are allowed, the same pairwise sign information no longer prevents stable coexistence (Theorem~\ref{thm: higher stabilization}). More precisely, for arbitrary prescribed growth rates and pairwise coefficients, we construct higher-order interaction terms so that the resulting system has a feasible asymptotically stable equilibrium. Thus, pairwise impossibility need not persist once genuinely higher-order interactions are introduced.

The final part of the paper develops the connection with evolutionary game theory. Building on the higher-order Lotka--Volterra--replicator correspondence outlined in \cite{Gokhalev2}, we formulate it for arbitrary interaction order using tensors symmetric in their partner indices and study the associated symmetric multiplayer games. Extensions of pairwise-interactions dynamical systems to higher-order interactions \cite{abiad2026hypergraphs,battiston2020networks,bick2023higher} have attracted considerable interest recently with generalized Lotka--Volterra systems as a prominent example \cite{JardonLotkaTensors, gibbs2024can,letten2019mechanistic,levine2017beyond}. In the context of the Lotka--Volterra–replicator correspondence, feasible equilibria correspond to totally mixed symmetric Nash equilibria, and we study their robustness under small perturbations of the game. In the pairwise setting, each impossible ecological sign pattern determines a nonempty open class of symmetric two-player games with no totally mixed evolutionarily stable strategy (Corollary~\ref{cor: formats without ess}). Thus, qualitative sign information on ecological interactions yields robust obstructions in payoff space, while the higher-order construction shows that such pairwise obstructions can disappear once multiplayer interactions are introduced.

The paper is organized as follows. In \Cref{sec:FSstratum}, we recall feasible equilibria and reduce the impossibility problem to the exclusion of Hurwitz-stable feasible equilibria. In \cref{sec:impossible-ecologies}, we prove the impossibility results for the eleven four-species ecologies and for the infinite families described above. In \cref{sec:higher}, we show how higher-order interactions can restore stable coexistence and develop the correspondence with higher-order replicator dynamics and symmetric games, including the consequences for payoff robustness and evolutionarily stable strategies.

\section{Stable Coexistence in the Generalized Lotka--Volterra}\label{sec:FSstratum}

Stable coexistence in population dynamics is modeled by an equilibrium at which all species are present and which is stable under small perturbations. In this section, we review the framework introduced in \cite{Haas2025} for studying the existence of such equilibria from qualitative information on the growth rates and interaction coefficients, namely their signs. This leads to a feasibility-and-stability problem for Lotka--Volterra systems.\\

\noindent The generalized Lotka--Volterra system (GLV) is the dynamical system
\begin{equation}\label{eq: lv}
    \dot{\bx} = \diag(\bx)(\ba - \bB\bx),
\end{equation}
where $\bx = (x_i) \in \R^n$ is a vector-valued function of time $t$,  $ \dot{\bx}=(\dot{x}_i)$ represents the vector of time derivatives $\dot{x}_i$ of $x_i$, $\ba = (a_i) \in \R^n$, and $\bB = (b_{ij}) \in \R^{n\times n}$. In components, this is the system of ordinary differential equations (ODEs):
\[
    \dot{x}_i = x_i \left( a_i - \sum_{j=1}^{n} b_{ij} x_j \right), \quad \text{for } i \in [n].
\]
This system plays an important role in theoretical ecology as it models pairwise interactions of $n$ species in the same ecosystem. The parameter $a_i$ represents the intrinsic growth rate of species $i$ in the absence of other interactions. Thus, $a_i > 0$ for a species that grows on its own, and $a_i < 0$ for a species that dies out in the absence of others. We also assume $b_{ii} > 0$, meaning each species has a finite carrying capacity. The sign of $b_{ij}$ (for $i \neq j$) models the specific type of interaction between species $i$ and $j$: \emph{competitive} if $b_{ij},b_{ji}>0$, \emph{mutualistic} if $b_{ij},b_{ji}<0$ or \emph{prey-predator interaction}, with species $i$ predating species $j$, if $b_{ij}<0$, $b_{ji}>0$. In short, one can naturally associate to the parameters $\ba = (a_i) \in \R^n$, and $\bB = (b_{ij}) \in \R^{n\times n}$ the \emph{ordered sign pattern}:  $$\sign(\ba,\bB)=(\sign(a_1),\ldots, \sign(a_n),\sign(b_{11}),\sign(b_{12}),\ldots,\sign(b_{n-1,n}),\sign(b_{n,n}))\in \{\pm\}^{n+n^2}.$$ Thus, the signs of the growth rates and interaction coefficients determine the associated \emph{ecological network} that is illustrated in \cref{fig: Pierregraph}.

\begin{figure}[h!]
\centering
\tikzset{
  fillednode/.style={circle, draw=black, fill=black, inner sep=2pt},
  node/.style={circle, draw=black, fill=white, inner sep=2pt},
  edge/.style={line width=1pt},
  redge/.style={edge, red},
  bedge/.style={edge, blue}
}
\noindent\begin{tikzpicture}[>=stealth,scale=1.3]
\draw[postaction={decorate},decoration={markings,mark=at position 0.50 with {\arrow{<}}},thick] (0.3090,0.9511) -- (-0.8090,0.5878);
\draw[postaction={decorate},decoration={markings,mark=at position 0.80 with {\arrow{<}}},thick] (0.3090,0.9511) -- (-0.8090,-0.5878);
\draw[red,thick] (0.3090,0.9511) -- (0.3090,-0.9511);
\draw[red,thick] (0.3090,0.9511) -- (1.0000,-0.0000);
\draw[blue,thick] (-0.8090,0.5878) -- (-0.8090,-0.5878);
\draw[postaction={decorate},decoration={markings,mark=at position 0.20 with {\arrow{>}}},thick] (-0.8090,0.5878) -- (0.3090,-0.9511);
\draw[postaction={decorate},decoration={markings,mark=at position 0.20 with {\arrow{>}}},thick] (-0.8090,0.5878) -- (1.0000,-0.0000);
\draw[postaction={decorate},decoration={markings,mark=at position 0.50 with {\arrow{>}}},thick] (-0.8090,-0.5878) -- (0.3090,-0.9511);
\draw[postaction={decorate},decoration={markings,mark=at position 0.20 with {\arrow{>}}},thick] (-0.8090,-0.5878) -- (1.0000,-0.0000);
\draw[blue,thick] (0.3090,-0.9511) -- (1.0000,-0.0000);
\fill (0.3090,0.9511) circle (0.10);
\fill (-0.8090,0.5878) circle (0.10);
\fill (-0.8090,-0.5878) circle (0.10);
\filldraw[fill=white] (0.3090,-0.9511) circle (0.10);
\filldraw[fill=white] (1.0000,-0.0000) circle (0.10);
\end{tikzpicture}
\vspace{.5cm}
\hspace{1cm}
\noindent\begin{tikzpicture}[every text node part/.style={align=center},line cap=round,>=stealth,font=\small,scale=1.3]
\node at (-0.08,-1.5) {${\dot{\bx}}=\diag(\bx)\left(\ba-\bB\cdot\bx\right)$};
\draw[<-] (-1.05,-1.65) -- (-1.05,-2.2) -- (-0.5,-2.2) node[anchor=west,inner sep=1.5pt] {$\displaystyle\dot{x}_i=x_i\left(a_i-\sum_{j=1}^n{b_{ij}x_j}\right),\;i \in [n]$};
\draw[<-] (0.9,-2.4) -- (0.9,-2.9) -- (-0.2,-2.9) -- (-0.2,-3.2) node[anchor=north] {growth rates\\$a_i\gtrless 0$};
\draw[<-] (2,-2.4) -- (2,-3.2) node[anchor=north] {interactions\\$b_{ij}\gtrless 0,b_{ii}>0$};
\end{tikzpicture}

\noindent\begin{tikzpicture}[every text node part/.style={align=center},line cap=round,>=stealth,font=\small,scale=1.3]
\begin{scope}[shift={(0,0.2)}]
\filldraw (0,-2) circle(2.5pt);
\node at (-0.25,-2) {$i\vphantom{j}$};
\node[anchor=west] at (0.4,-2) {growing species:\vphantom{d}};
\node at (1.7,-2.4) {$a_i > 0 $};
\node at (-0.25,-2.9) {$i\vphantom{j}$};
\draw (0,-2.9) circle (2.5pt);
\node[anchor=west] at (0.4,-2.9) {dying species:};
\node at(1.7,-3.3) {$a_i < 0 $};
\end{scope}
\begin{scope}[shift={(3.3,2)}]
\draw[red,thick] (0.1,-3.8) node[anchor=east,inner sep=1.5pt,black] {$i\vphantom{j}$} -- (0.6,-3.8) node[anchor=west,inner sep=1.5pt,black] {$j$};
\draw[blue,thick] (0.1,-4.5) node[anchor=east,inner sep=1.5pt,black] {$i\vphantom{j}$} -- (0.6,-4.5) node[anchor=west,inner sep=1.5pt,black] {$j$};
\draw[postaction={decorate},decoration={markings,mark=at position 0.6 with {\arrow{>}}},thick] (0.1,-5.2) node[anchor=east,inner sep=1.5pt,black] {$i\vphantom{j}$} -- (0.6,-5.2) node[anchor=west,inner sep=1.5pt,black] {$j$};
\node[anchor=west] at (1,-3.8){competition:};
\node at (3.6,-3.8) {$b_{ij},b_{ji}>0$};
\node[anchor=west] at (1,-4.5) {mutualism:\vphantom{p}};
\node at (3.6,-4.5){$b_{ij},b_{ji} < 0$};
\node[anchor=west] at (1,-5.2) {predation:};
\node at (3.6,-5.2){$b_{ij} < 0,b_{ji}>0$};
\end{scope}
\end{tikzpicture}

\caption{Lotka--Volterra ecological dynamics on a network of $n=5$ species: definition of the different types of ecological interactions. Figure redrawn from \cite{Haas2025}.
\label{fig: Pierregraph}}
\end{figure}
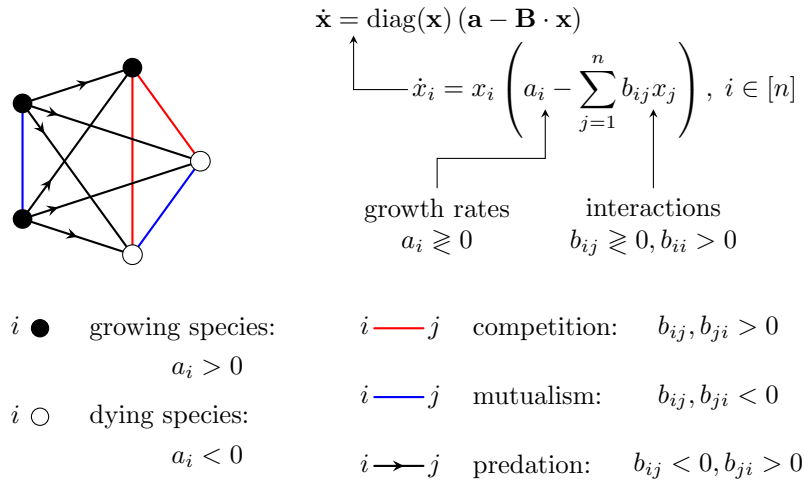

At a steady state of the system, $\dot{\bx}=0$. If $\bB$ is
invertible, the unique equilibrium with all coordinates nonzero is ${\xstar}=\bB^{-1}\ba$. We call this equilibrium \emph{feasible} if ${\xstar}>0$, corresponding
to the coexistence of all species.

Following \cite{Haas2025}, we are interested in feasible equilibria
that are \emph{locally asymptotically stable}, meaning that solutions
starting sufficiently close to the equilibrium remain close and
converge to it. We are interested not in a single
choice of parameters, but in what can be inferred from their signs
alone. Thus, for a sign pattern
$\sigma\in\{\pm\}^{n+n^2}$, or equivalently an ecological network as
in \cref{fig: Pierregraph}, we say that the associated ecological
dynamics is \emph{possible} if there exist parameters
$\ba\in\R^n$ and $\bB\in\R^{n\times n}$ with
$\sign(\ba,\bB)=\sigma$ such that the Lotka--Volterra system
\labelcref{eq: lv} admits a feasible locally asymptotically stable
equilibrium. Otherwise, we call it \emph{impossible}.

For the Lotka--Volterra system, the Jacobian at a feasible equilibrium
is $J_{\ba,\bB}({\xstar})=-\diag({\xstar})\bB$.
If this matrix is Hurwitz, i.e., all its eigenvalues have strictly negative real part, then ${\xstar}$ is locally asymptotically
stable. The converse need not hold, since a locally asymptotically
stable equilibrium may have Jacobian eigenvalues with zero real part.
The following observation shows that such marginal cases do not
affect the classification by sign patterns: whenever a feasible
equilibrium has no Jacobian eigenvalue with positive real part, the
parameters can be perturbed within the same sign pattern so that the
Jacobian becomes Hurwitz.

\begin{proposition}\label{prop: stability under perturbation}
Let ${\xstar}=\bB^{-1}\ba>0$ be a feasible equilibrium of the GLV system, and suppose that all eigenvalues of $\JacAB({\xstar})=-\diag({\xstar})\bB$ have non-positive real part. Then, for every sufficiently small $\varepsilon>0$, there are perturbed parameters $\widetilde{\bB}=\bB+\varepsilon\diag({\xstar})^{-1}$ and  $\widetilde{\ba}=\ba+\varepsilon\mathbf 1$
with the same sign pattern as $(\ba,\bB)$, such that ${\xstar}$ is still a feasible equilibrium of the perturbed system and $J_{\widetilde{\ba},\widetilde{\bB}}({\xstar})=\JacAB({\xstar})-\varepsilon\bI$.
In particular, $J_{\widetilde{\ba},\widetilde{\bB}}({\xstar})$ is Hurwitz, and hence ${\xstar}$ is locally asymptotically stable for the perturbed system.
\end{proposition}

\begin{proof}
Since ${\xstar}>0$, the diagonal matrix $\diag({\xstar})^{-1}$ is well-defined. Define $\widetilde{\bB}$ and $\widetilde{\ba}$ as above. Then $
\widetilde{\bB}{\xstar}
=
\bB{\xstar}+\varepsilon\diag({\xstar})^{-1}{\xstar}
=
\ba+\varepsilon\mathbf 1
=
\widetilde{\ba}$.
Thus ${\xstar}$ is still a feasible equilibrium of the perturbed system. The Jacobian of the perturbed system at ${\xstar}$ is
\[
J_{\widetilde{\ba},\widetilde{\bB}}({\xstar})
=
-\diag({\xstar})\widetilde{\bB}
=
-\diag({\xstar})\bB-\varepsilon\diag({\xstar})\diag({\xstar})^{-1}
=
\Jx-\varepsilon\bI.
\]
Therefore, if $\lambda_1,\dots,\lambda_n$ are the eigenvalues of $\Jx$, then the eigenvalues of $J_{\widetilde{\ba},\widetilde{\bB}}({\xstar})$ are $\lambda_1-\varepsilon,\dots,\lambda_n-\varepsilon$. Since $\operatorname{Re}(\lambda_i)\leq 0$ for every $i$, we have $\operatorname{Re}(\lambda_i-\varepsilon)<0$ for every $i$.

It remains to check that the sign pattern is preserved. The off-diagonal entries of $\bB$ are unchanged. The diagonal entries change by adding the positive quantity $\varepsilon/x_i^\star$; since $b_{ii}>0$, their signs remain positive. Finally, $\widetilde a_i=a_i+\varepsilon$. Since the sign pattern records only non-zero signs, and hence $a_i\neq 0$ for every $i$, choosing $\varepsilon>0$ sufficiently small preserves the sign of every $a_i$. Therefore $(\widetilde{\ba},\widetilde{\bB})$ has the same sign pattern as $(\ba,\bB)$.
\end{proof} \vspace{0.1cm}

Consequently, a sign pattern is impossible if and only if no choice of parameter magnitudes with that sign pattern yields a feasible equilibrium whose Jacobian is Hurwitz. This leads to the following problem.

\begin{problem}\label{ProbImpossibleEco}
Given the ordered sign pattern $\sigma = (\sigma_1,\dots,\sigma_n,\sigma_{11},\dots,\sigma_{nn}) \in \{\pm\}^{n+n^2}$, namely an ecological network in \cref{fig: Pierregraph}, decide whether there is no $({\bf a},{\bf B})\in \mathbb{R}^n \times \mathbb{R}^{n^2} $ with $\sign(\ba,\bB) = \sigma$ for which ${\xstar}=\bB^{-1}\ba>0$ and the matrix $\Jx=-\diag({\xstar})\bB$ is Hurwitz.
\end{problem}

\begin{example}[Two-species obligate mutualism]
\label{ex:two-species-obligate-mutualism}
Consider two species in obligate mutualism, so that $a_1,a_2<0$, $b_{11},b_{22}>0$, and $b_{12},b_{21}<0$.
Suppose that there exists a feasible equilibrium
${\xstar}=\bB^{-1}\ba>0$ whose Jacobian $\Jx=-\diag({\xstar})\bB$ is Hurwitz. Then $0<\det(\Jx)=x_1^\star x_2^\star\det(\bB)$, and feasibility implies $\det(\bB)>0$. On the other hand, Cramer's
rule gives $x_1^\star = (a_1b_{22}-b_{12}a_2)/{\det(\bB)}<0$. This contradicts
feasibility. 
\end{example}

This example makes visible that the impossible-ecology problem becomes an algebraic compatibility problem between sign constraints, feasibility inequalities, and spectral conditions. The problem is completely understood for two and three species. In the two-species case, among the six non-trivial interaction patterns, only obligate mutualism is impossible. For three species, the four non-trivially impossible ecologies are competition with two obligate mutualists, facultative predation on two obligate mutualists, obligate cyclic predation, and obligate mutualism of three species; see \cite{Haas2025}. For larger systems, \cite{Haas2025} combined a classification of interaction topologies with numerical sampling of parameter magnitudes. This led to eleven candidate impossible ecologies for four species, displayed in \cref{fig:symmetricimpossible,fig: asymmetricimpossible}, and to further candidates for five species. These larger cases were earlier left conjectural.

\begin{figure}[h]
\centering

\tikzset{
  fillednode/.style={circle, draw=black, fill=black, inner sep=2pt},
  node/.style={circle, draw=black, fill=white, inner sep=2pt},
  edge/.style={line width=1pt},
  redge/.style={edge, red},
  bedge/.style={edge, blue},
  >=stealth
}

\foreach \i in {0,...,4} {
  \begin{tikzpicture}[scale=0.8, baseline={(current bounding box.center)}]
    \coordinate (A) at (0,2);
    \coordinate (B) at (-1,1);
    \coordinate (C) at (1,1);
    \coordinate (D) at (0,0);

    \ifnum\i=0
      \draw[bedge] (A) -- (B);
      \draw[bedge] (A) -- (C);
      \draw[bedge] (A) -- (D);
      \draw[bedge] (B) -- (D);
      \draw[bedge] (B) -- (C);
      \draw[bedge] (C) -- (D);
    \else\ifnum\i=1
      \draw[redge] (A) -- (B);
      \draw[redge] (A) -- (C);
      \draw[redge] (A) -- (D);
      \draw[bedge] (B) -- (D);
      \draw[bedge] (B) -- (C);
      \draw[bedge] (C) -- (D);
    \else
      \draw[bedge] (A) -- (B);
      \draw[redge] (A) -- (C);
      \draw[redge] (A) -- (D);
      \draw[redge] (B) -- (D);
      \draw[redge] (B) -- (C);
      \draw[bedge] (C) -- (D);
    \fi\fi

    \ifnum\i=0
      \node[node] at (A) {};
      \node[node] at (B) {};
      \node[node] at (C) {};
      \node[node] at (D) {};
    \else\ifnum\i=1
      \node[fillednode] at (A) {};
      \node[node] at (B) {};
      \node[node] at (C) {};
      \node[node] at (D) {};
    \else\ifnum\i=2
      \node[fillednode] at (A) {};
      \node[fillednode] at (B) {};
      \node[node] at (C) {};
      \node[node] at (D) {};
    \else\ifnum\i=3
      \node[fillednode] at (A) {};
      \node[node] at (B) {};
      \node[node] at (C) {};
      \node[node] at (D) {};
    \else
      \node[node] at (A) {};
      \node[node] at (B) {};
      \node[node] at (C) {};
      \node[node] at (D) {};
    \fi\fi\fi\fi
    \node at (0,-0.8) {(\char\numexpr`a+\i\relax)};  
  \end{tikzpicture}
  \hspace{0.6cm}
}
\caption{The five symmetric four-species ecologies conjectured impossible in \cite{Haas2025}.}
\label{fig:symmetricimpossible}
\end{figure}
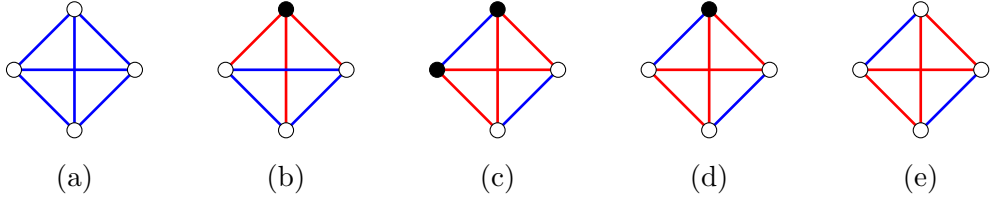

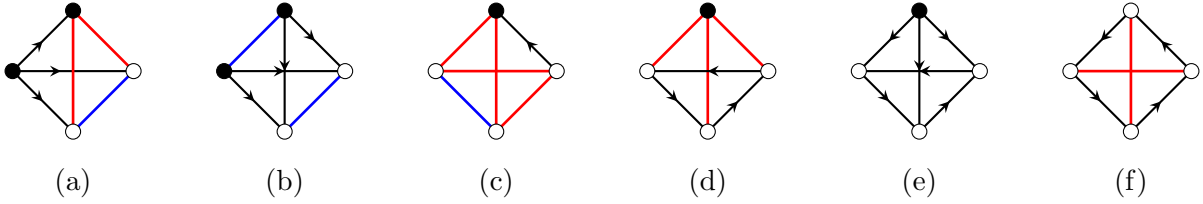
\begin{figure}[h]
\centering

\tikzset{
  fillednode/.style={circle, draw=black, fill=black, inner sep=2pt},
  node/.style={circle, draw=black, fill=white, inner sep=2pt},
  edge/.style={line width=1pt},
  redge/.style={edge, red},
  bedge/.style={edge, blue},
  midarrow/.style={
    postaction={decorate},
    decoration={markings, mark=at position 0.50 with {\arrow{>}}}
  },
  >=stealth
}

\foreach \i in {0,...,5} {
  \begin{tikzpicture}[scale=0.8, baseline={(current bounding box.center)}]
    \coordinate (A) at (0,1);
    \coordinate (B) at (-1,0);
    \coordinate (C) at (1,0);
    \coordinate (D) at (0,-1);

    \ifnum\i=0
      \draw[postaction={decorate},decoration={markings,mark=at position 0.50 with {\arrow{>}}},thick] (B) -- (A);
      \draw[postaction={decorate},decoration={markings,mark=at position 0.40 with {\arrow{>}}},thick] (B) -- (C);
      \draw[midarrow, thick] (B) -- (D);
      \draw[redge] (A) -- (D);
      \draw[redge] (A) -- (C);
      \draw[bedge] (C) -- (D);
    \else\ifnum\i=1
      \draw[bedge] (A) -- (B);
      \draw[midarrow, thick] (A) -- (C);
      \draw[midarrow, thick] (A) -- (D);
      \draw[midarrow, thick] (B) -- (C);
      \draw[midarrow, thick] (B) -- (D);
      \draw[bedge] (C) -- (D);
    \else\ifnum\i=2
      \draw[redge] (B) -- (A);
      \draw[midarrow, thick] (C) -- (A);
      \draw[redge] (D) -- (A);
      \draw[redge] (B) -- (C);
      \draw[bedge] (B) -- (D);
      \draw[redge] (C) -- (D);
    \else\ifnum\i=3
      \draw[redge] (A) -- (B);
      \draw[redge] (A) -- (C);
      \draw[redge] (A) -- (D);
      \draw[midarrow, thick] (B) -- (D);
      \draw[midarrow, thick] (C) -- (B);
      \draw[midarrow, thick] (D) -- (C);
    \else\ifnum\i=4
      \draw[midarrow, thick] (A) -- (C);
      \draw[midarrow, thick] (C) -- (B);
      \draw[midarrow, thick] (D) -- (C);
      \draw[midarrow, thick] (B) -- (D);
      \draw[midarrow, thick] (A) -- (D);
      \draw[midarrow, thick] (A) -- (B);
    \else
      \draw[midarrow, thick] (A) -- (B);
      \draw[midarrow, thick] (C) -- (A);
      \draw[redge] (A) -- (D);
      \draw[midarrow, thick] (B) -- (D);
      \draw[redge] (B) -- (C);
      \draw[midarrow, thick] (D) -- (C);
    \fi\fi\fi\fi\fi

    \ifnum\i=0
      \node[fillednode] at (A) {};
      \node[fillednode] at (B) {};
      \node[node] at (C) {};
      \node[node] at (D) {};
    \else\ifnum\i=1
      \node[fillednode] at (A) {};
      \node[fillednode] at (B) {};
      \node[node] at (C) {};
      \node[node] at (D) {};
    \else\ifnum\i=2
      \node[fillednode] at (A) {};
      \node[node] at (B) {};
      \node[node] at (C) {};
      \node[node] at (D) {};
    \else\ifnum\i=3
      \node[fillednode] at (A) {};
      \node[node] at (B) {};
      \node[node] at (C) {};
      \node[node] at (D) {};
    \else\ifnum\i=4
      \node[fillednode] at (A) {};
      \node[node] at (B) {};
      \node[node] at (C) {};
      \node[node] at (D) {};
    \else
      \node[node] at (A) {};
      \node[node] at (B) {};
      \node[node] at (C) {};
      \node[node] at (D) {};
    \fi\fi\fi\fi\fi
    \node at (0,-1.85) {(\char\numexpr`a+\i\relax)};
  \end{tikzpicture}
  \hspace{0.6cm}
}
\caption{The asymmetric four-species ecologies conjectured to be impossible in \cite{Haas2025}.}
\label{fig: asymmetricimpossible}
\end{figure}

A subsequent approach in \cite{Celik2025} reformulated \cref{ProbImpossibleEco} as a non-emptiness problem for semialgebraic sets and transferred part of the analysis to real Grassmannians. Using combinatorial and computational methods based on Grassmann--Pl\"ucker relations in the Grassmannian, the authors proved the impossibility of three of the eleven four-species candidates. 

In the next section, we prove that all eleven four-species candidates are impossible. Several of these arguments extend beyond the four-species setting and yield infinite families of impossible ecological interaction patterns (Section~\ref{sec: 3.1}, \ref{sec: 3.2}). Our arguments use tools from matrix analysis and spectral theory, including $M$-matrix methods and Schur-complement reductions. We recall the required background in the following section.

\subsection{Tools from Matrix Analysis}\label{sec:matrixtools}

A matrix $\bB=(b_{ij})\in\R^{n\times n}$ is a $Z$-\emph{matrix} if all its off-diagonal entries are non-positive, that is, $b_{ij}\leq 0$ for all $i\neq j$. A family of $Z$-matrices of particular interest in many domains is the class of $M$-matrices. A matrix $\bB\in\R^{n\times n}$ is an \emph{$M$-matrix} if it is a $Z$-matrix and all its eigenvalues have strictly positive real parts. 

We will use the following two characterizations of nonsingular $M$-matrices. More than forty equivalent characterizations are known (see, for instance, \cite[Theorem 1]{Plemmons1977}). We record only the two forms needed below. Inequalities between matrices and vectors are understood entrywise.

\begin{proposition}\label{prop:characterizations-M-matrix}

Let $M$ be a $Z$-matrix. Then the following are equivalent:

\begin{enumerate}

    \item $M$ is a nonsingular $M$-matrix;

    \item $M^{-1}\geq 0$;

    \item all principal minors of $M$ are strictly positive.

\end{enumerate}

\end{proposition}

We continue here recalling the notion of \emph{Schur complement} of a matrix, an important matrix operation, which we will use repeatedly later in this section. For any matrix $\bB\in\R^{n\times n}$ and any indices $P,Q$ with entries in $[n]$ and length $k<n$ and $s<n$, we denote by $\bB_{PQ}$ the submatrix with rows indexed by $P$ and columns indexed by $Q$. We recall also that for an index set $Q\subset[n]$ of length $k$ we denote with $Q^c$ the complement index set sorted in increasing order. If the submatrix $\bB_{QQ}$ is non-singular, the Schur complement of $\bB_{QQ}$ in $\bB$, denoted by $(\bB / \bB_{QQ})$, is the $(n - k) \times (n - k)$ matrix defined by $(\bB / \bB_{QQ}) = \bB_{Q^cQ^c} - \bB_{Q^cQ} \bB_{QQ}^{-1} \bB_{QQ^c}$.\vspace{0.1cm}

\noindent We will repeatedly use the following standard determinant identity for Schur complements; see \cite[Section~0.8.5]{horn2012matrix}.

\begin{proposition}\label{lem: schur}
Let $Q\subseteq [n]$ and $\bB\in\R^{n\times n}$ be a square matrix with non-singular principal submatrix $\bB_{QQ}$. The following equality holds: $\det \bB=\det(\bB_{QQ})\det  (\bB / \bB_{QQ})$. In particular, if $\det(\bB_{QQ})>0$ and $\det  (\bB / \bB_{QQ})<0$, then $\det \bB<0$.
\end{proposition}

\section{Results on Impossible Ecologies}\label{sec:impossible-ecologies}

In this section, we prove the impossibility of all eleven four-species
ecological networks conjectured in \cite{Haas2025}. Several of the
arguments extend to arbitrary numbers of species and yield infinite
families of impossible ecologies.

We first treat the five symmetric networks in
\cref{fig:symmetricimpossible} using $M$-matrix methods. We then study
families built from obligate prey--predator cycles with inner
competition, which include three of the asymmetric four-species
networks. The remaining three asymmetric networks are handled
individually by Schur-complement arguments. Together, these results
complete the classification of the eleven four-species candidates.

\subsection{Symmetric Ecologies}\label{sec: 3.1}

Among the $11$ ecologies of four species conjectured impossible in \cite{Haas2025}, five are symmetric, meaning that every pairwise interaction is either mutualistic or competitive. These five ecological networks are displayed in \cref{fig:symmetricimpossible}. We prove below that they are not isolated four-dimensional phenomena, but they are instances of two general obstruction mechanisms valid for any number of species.

We start with studying the ecological pattern represented by the ecological network (a) in \cref{fig:symmetricimpossible} in generality. Recall the Lotka--Volterra system \eqref{eq: lv} with parameters $\ba=(a_i)\in \R^n$ and $\bB=(b_{ij})\in \R^{n \times n}$. An ecological network is in \emph{obligate mutualism} if $a_i<0$ for every $i\in[n]$ and $b_{ij}<0$ for every $i,j \in [n]$ with $i\neq j$. Such an ecological network for $n=5$ is represented in \cref{fig: impsymmecoanyn}. This case turns out to be impossible in arbitrary number of species. In order to show this we use the properties of the inverse of an $M$-matrix as listed in \cref{prop:characterizations-M-matrix}. 

\begin{theorem}\label{thm: obligate mutualism}
For any $n\geq 2$, every ecological network in obligate mutualism is impossible.
\end{theorem}

\begin{proof}
Assume by contradiction that the associated ecological network is possible, then ${\xstar}=\bB^{-1}\ba>0$ is such that $\Jx$ has only eigenvalues with negative real part. Since $\Jx=-\diag({\xstar})\bB$, we have that the signs of the entries of the matrix $\Jx$ are the same as the signs of the matrix $-\bB$ as $\diag({\xstar})$ is a diagonal matrix with positive elements on the diagonal. Therefore $-\Jx$ is also a $Z$-matrix as $\bB$ is a $Z$-matrix. Moreover, by assumption $\Jx$ has only eigenvalues with negative real part, which implies that $-\Jx$ has only eigenvalues with positive real part. Hence $-\Jx$ is a non-singular $M$-matrix. By \cref{prop:characterizations-M-matrix}, $(-\Jx)^{-1}=\bB^{-1}\diag({\xstar})^{-1}\geq 0$, where we recall that the inequality $\geq 0$ for matrices means entrywise non-negativity. Hence $\bB^{-1}=\bB^{-1}\diag({\xstar})^{-1}\diag({\xstar})\geq 0$. Since $\ba<\mathbf 0$, this forces ${\xstar}=\bB^{-1}\ba\leq \mathbf 0$, contradicting feasibility. Therefore, the ecological network is impossible.
\end{proof} \vspace{0.1cm}

We now focus on the remaining four ecological patterns represented in \cref{fig:symmetricimpossible}. These are all instances of a specific class of ecological networks as in the following. Let $[n]=I\sqcup J$ be a nontrivial partition. An ecological network consists of \emph{two competing mutualistic} groups if species within the same group interact mutualistically, while species in different groups compete; explicitly, $b_{ij}<0$ whenever $i\neq j$ and either $i,j\in I$ or $i,j\in J$, and $b_{ij}>0$ whenever $i\in I$, $j\in J$ or $i\in J$, $j\in I$. Moreover, the ecological network consisting of two competing mutualistic groups has one group in obligate mutualism if $a_i<0$ for all $i\in I$ or for all $i\in J$. An example of an ecological network $n=7$ species with two competing mutualistic groups and one group in obligate mutualism is given in \cref{fig: impsymmecoanyn}.

\begin{figure}[h]
\centering

\tikzset{
  fillednode/.style={circle, draw=black, fill=black, inner sep=2.5pt},
  emptynode/.style={circle, draw=black, fill=white, inner sep=2.5pt},
  bedge/.style={blue, line width=1pt},
  redge/.style={red, line width=1pt},
  group/.style={draw=black!40, rounded corners, dashed, inner sep=8pt}
}

\begin{minipage}{0.45\textwidth}
\centering
\begin{tikzpicture}[scale=0.82]

\coordinate (A) at (0,2);
\coordinate (B) at (-1.35,1);
\coordinate (C) at (1.35,1);
\coordinate (D) at (-0.85,-0.65);
\coordinate (E) at (0.85,-0.65);

\foreach \a/\b in {A/B,A/C,A/D,A/E,B/C,B/D,B/E,C/D,C/E,D/E}{
  \draw[bedge] (\a) -- (\b);
}

\node[emptynode] at (A) {};
\node[emptynode] at (B) {};
\node[emptynode] at (C) {};
\node[emptynode] at (D) {};
\node[emptynode] at (E) {};

\end{tikzpicture}

\end{minipage}
\hfill
\begin{minipage}{0.52\textwidth}
\centering
\begin{tikzpicture}[scale=0.72]

\coordinate (I1) at (-1.15,2);
\coordinate (I2) at (0,2.50);
\coordinate (I3) at (1.15,2);

\coordinate (J1) at (-1.6,-1.5);
\coordinate (J2) at (-0.5,-2.25);
\coordinate (J3) at (0.8,-2.05);
\coordinate (J4) at (1.6,-1.4);

\node[group, fit=(I1)(I2)(I3)] {};
\node[group, fit=(J1)(J2)(J3)(J4)] {};

\foreach \a/\b in {I1/I2,I1/I3,I2/I3}{
  \draw[bedge] (\a) -- (\b);
}

\foreach \a/\b in {J1/J2,J1/J3,J1/J4,J2/J3,J2/J4,J3/J4}{
  \draw[bedge] (\a) -- (\b);
}

\foreach \i in {I1,I2,I3}{
  \foreach \j in {J1,J2,J3,J4}{
    \draw[redge, opacity=0.55] (\i) -- (\j);
  }
}

\node[fillednode] at (I1) {};
\node[emptynode] at (I2) {};
\node[fillednode] at (I3) {};

\node[emptynode] at (J1) {};
\node[emptynode] at (J2) {};
\node[emptynode] at (J3) {};
\node[emptynode] at (J4) {};

\node at (2.1,2.35) {$I$};
\node at (2.5,-1.95) {$J$};

\end{tikzpicture}

\end{minipage}

\caption{Examples of obligate mutualism for $n=5$ (on the left) and two competing mutualistic groups for $n=7$ (on the right).}
\label{fig: impsymmecoanyn}
\end{figure}
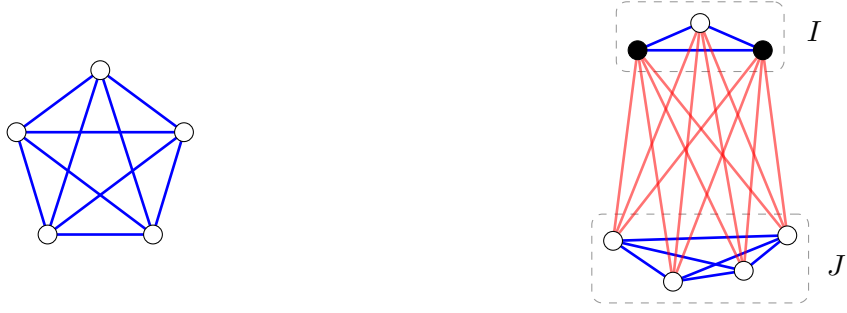

The following theorem shows that this entire class of ecological networks is impossible.

\begin{theorem}\label{thm: competing mutualism}
For any $n\geq 2$, every ecological network consisting of two competing mutualistic groups, at least one of which is in obligate mutualism, is impossible.
\end{theorem}
\begin{proof}
Let $[n]=I\sqcup J$ be the partition of the species into the two competing mutualistic groups. Without loss of generality, assume that the group $J$ is in obligate mutualism, e.g., as in \cref{fig: impsymmecoanyn}. Assume by contradiction that the associated ecological network is possible, then ${\xstar}=\bB^{-1}\ba>0$ and the matrix $\Jx $ has only eigenvalues with negative real part. Let $\bS$ be the $n \times n$ diagonal matrix defined by $(\bS)_{kk} = 1$, if $k \in I$ and $(\bS)_{kk} = -1$, if $k \in J$. Notice that $\bS^{-1}=\bS$. Consider the matrix $\bS \bB \bS$. We analyze its off-diagonal entries $(\bS \bB \bS)_{kl}$ for $k \neq l$:
\begin{itemize}
    \item If $k, l \in I$, $(\bS \bB \bS)_{kl} = 1\cdot b_{kl} \cdot 1 = b_{kl} < 0$ (by mutualism within $I$).
    \item If $k, l \in J$, $(\bS \bB \bS)_{kl} = -1 \cdot b_{kl} \cdot -1 = b_{kl} < 0$ (by mutualism within $J$).
    \item If $k \in I$ and $l \in J$ (or vice versa), $(\bS \bB \bS)_{kl} = 1 \cdot b_{kl} \cdot (-1) = -b_{kl} < 0$ (by competition between $I$ and $J$).
\end{itemize}
Thus, all off-diagonal entries of $\bS \bB \bS$ are strictly negative, making it a $Z$-matrix. We now apply the similarity transformation induced by $\bS$ to $-\Jx$ and, using commutativity of diagonal matrices, we obtain
\begin{equation*}\label{eq: 2competingmutialisticgroups}
    -\bS \Jx \bS = \bS \diag({\xstar}) \bB \bS = \diag({\xstar}) (\bS \bB \bS).
\end{equation*}
Therefore, -$\bS \Jx \bS$ is a $Z$-matrix as $ \diag({\xstar})$ is a diagonal matrix with positive diagonal and $\bS \bB \bS$ is a $Z-$matrix. The matrix $-\bS \Jx \bS$ has eigenvalues with strictly positive real parts as it is similar to the matrix $-\Jx$ and thus they have the same eigenvalues. Therefore, $-\bS \Jx \bS$ is a non-singular $M$-matrix. By \cref{prop:characterizations-M-matrix} all principal minors of $-\bS \Jx \bS$ are strictly positive. Moreover, since $\bS$ and $\diag({\xstar})$ are diagonal matrices, any principal minor of $\diag({\xstar})\bS \bB \bS$ indexed by $K \subseteq [n]$ is given by $\det(\diag({\xstar})_{KK}\bS_{KK} \bB_{KK} \bS_{KK}) = \det(\diag({\xstar})_{KK})\det(\bS_{KK})^2 \det(\bB_{KK}) = \det(\bB_{KK})\prod_{k\in K}({\xstar})_k$. Therefore, we obtain that all principal minors of $\bB$ are strictly positive as ${\xstar}>0$. 
Now, consider the equilibrium equations $\bB {\xstar} = \ba$ restricted to the obligate mutualism group $J$. For any $j \in J$, we have 
\begin{equation}\label{eq: contradicting 2compmut}
    \sum_{k \in J} b_{jk} ({\xstar})_k + \sum_{i \in I} b_{ji} ({\xstar})_i = a_j.
\end{equation}
Since $J$ is in obligate mutualism ($a_j < 0$), the groups compete ($b_{ji} > 0$), and the equilibrium is feasible ($({\xstar})_i > 0$), from equation \labelcref{eq: contradicting 2compmut} we get $\sum_{k \in J} b_{jk} ({\xstar})_k < 0$ for all $j\in J$. In matrix notation, this reads $\bB_{JJ} ({\xstar})_J < 0$. We already established that $\bB$ has strictly positive principal minors, so its submatrix $\bB_{JJ}$ is a $Z$-matrix with positive principal minors. Therefore by \cref{prop:characterizations-M-matrix}, the matrix $\bB_{JJ}$ is itself a non-singular $M$-matrix. Using \cref{prop:characterizations-M-matrix} again, we obtain that $\bB_{JJ}^{-1} \ge 0$. Multiplying to the left the strict inequality $\bB_{JJ} ({\xstar})_J < 0$ by the non-negative matrix $\bB_{JJ}^{-1}$, gives $({\xstar})_J < 0$. This contradicts the feasibility assumption ${\xstar} > 0$, showing that the ecology is impossible.
\end{proof} \vspace{0.1cm}

\noindent Fixing $n=4$, \cref{thm: obligate mutualism} and  \cref{thm: competing mutualism} give the following corollary.

\begin{corollary}
    The five symmetric ecologies of \cref{fig:symmetricimpossible} are impossible.
\end{corollary}

\subsection{Impossibility of Prey-predator Cycles}\label{sec: 3.2}

We next study four families of ecological networks built from obligate prey-predator cycles with inner competition. Besides collections of such cycles, we allow either a single competing species, a self-sustaining super-predator, or an obligate mutualistic pair. In each case, the parity of the number of cycles determines an obstruction to stable coexistence.

\paragraph{Obligate cycle networks.} Let $n \geq 4$ and $k \geq 1$. We say that an ecological network is a collection of $k$ \emph{disjoint obligate prey-predator cycles with inner competition} if $a_i<0$ for every $i\in[n]$, $b_{ii}>0$ for every $i\in[n]$, and, after a relabeling of the species, they can be partitioned into $k$ disjoint directed cycles $C_1, \dots, C_k$, each of length at least $3$, such that the following conditions hold:
\begin{itemize}
    \item for every $i\in[n]$, if $i^+$ denotes the unique prey of $i$ in its directed cycle, then $b_{i,i^+}<0$ and $b_{i^+,i}>0$;
    \item all other pairs of species compete, namely $b_{ij}>0$ and $b_{ji}>0$ whenever $j\notin\{i,i^+\}$ and $i\neq j^+$.
\end{itemize}

\begin{figure}[htbp]
\centering

\tikzset{
  emptynode/.style={circle, draw=black, fill=white, inner sep=2pt},
  redge/.style={red, line width=0.8pt},
  cyclearrow/.style={
    thick,
    postaction={
      decorate,
      decoration={
        markings,
        mark=at position 0.55 with {\arrow{>}}
      }
    }
  },
  >=stealth
}

\begin{minipage}{0.31\textwidth}
\centering
\begin{tikzpicture}[scale=0.65]
  \coordinate (V1) at (90:1.6);    
  \coordinate (V2) at (162:1.6);   
  \coordinate (V3) at (234:1.6);   
  \coordinate (V4) at (306:1.6);   
  \coordinate (V5) at (18:1.6);    

  \foreach \a/\b in {1/4,1/5,2/3,2/4,3/5} {
    \draw[redge] (V\a) -- (V\b);
  }

  \draw[cyclearrow] (V1) -- (V3);
  \draw[cyclearrow] (V3) -- (V4);
  \draw[cyclearrow] (V4) -- (V5);
  \draw[cyclearrow] (V5) -- (V2);
  \draw[cyclearrow] (V2) -- (V1);

  \foreach \i in {1,...,5} {
    \node[emptynode] at (V\i) {};
  }
\end{tikzpicture}
\end{minipage}\hfill
\begin{minipage}{0.31\textwidth}
\centering
\begin{tikzpicture}[scale=0.65]
  \foreach \i/\ang in {1/90,2/30,3/-30,4/-90,5/-150,6/150} {
    \coordinate (V\i) at (\ang:1.6);
  }

  \foreach \a/\b in {1/3,1/4,1/5,2/4,2/5,2/6,3/5,3/6,4/6} {
    \draw[redge] (V\a) -- (V\b);
  }

  \draw[cyclearrow] (V1) -- (V2);
  \draw[cyclearrow] (V2) -- (V3);
  \draw[cyclearrow] (V3) -- (V4);
  \draw[cyclearrow] (V4) -- (V5);
  \draw[cyclearrow] (V5) -- (V6);
  \draw[cyclearrow] (V6) -- (V1);

  \foreach \i in {1,...,6} {
    \node[emptynode] at (V\i) {};
  }
\end{tikzpicture}
\end{minipage}\hfill
\begin{minipage}{0.31\textwidth}
\centering
\begin{tikzpicture}[scale=0.65]
  
  \path (0, 1.4) + (90:0.65) coordinate (V1);
  \path (0, 1.4) + (210:0.65) coordinate (V2);
  \path (0, 1.4) + (330:0.65) coordinate (V3);

  \path (-1.4, -0.9) + (90:0.65) coordinate (V4);
  \path (-1.4, -0.9) + (210:0.65) coordinate (V5);
  \path (-1.4, -0.9) + (330:0.65) coordinate (V6);

  \path (1.4, -0.9) + (135:0.75) coordinate (V7);
  \path (1.4, -0.9) + (225:0.75) coordinate (V8);
  \path (1.4, -0.9) + (315:0.75) coordinate (V9);
  \path (1.4, -0.9) + (45:0.75) coordinate (V10);

  \foreach \i in {1,2,3} {
    \foreach \j in {4,...,10} {
      \draw[redge] (V\i) -- (V\j);
    }
  }
  
  \foreach \i in {4,5,6} {
    \foreach \j in {7,...,10} {
      \draw[redge] (V\i) -- (V\j);
    }
  }
  
  \draw[redge] (V7) -- (V9);
  \draw[redge] (V8) -- (V10);

  \draw[cyclearrow] (V1) -- (V2);
  \draw[cyclearrow] (V2) -- (V3);
  \draw[cyclearrow] (V3) -- (V1);

  \draw[cyclearrow] (V4) -- (V5);
  \draw[cyclearrow] (V5) -- (V6);
  \draw[cyclearrow] (V6) -- (V4);

  \draw[cyclearrow] (V7) -- (V8);
  \draw[cyclearrow] (V8) -- (V9);
  \draw[cyclearrow] (V9) -- (V10);
  \draw[cyclearrow] (V10) -- (V7);

  \foreach \i in {1,...,10} {
    \node[emptynode] at (V\i) {};
  }
\end{tikzpicture}
\end{minipage}

\caption{Examples of obligate prey-predator cycles with inner competition. \textbf{Left and Center}: Single connected networks for $n=5$ and $n=6$, where all species form one directed cycle. \textbf{Right}: A network with $n=10$ species consisting of $k=3$ disjoint cycles (two of length 3, one of length 4). In all cases, species within a cycle follow a directed prey-predator loop, while all other pairs of species (including cross-cycle interactions) compete (red edges).}
\end{figure}
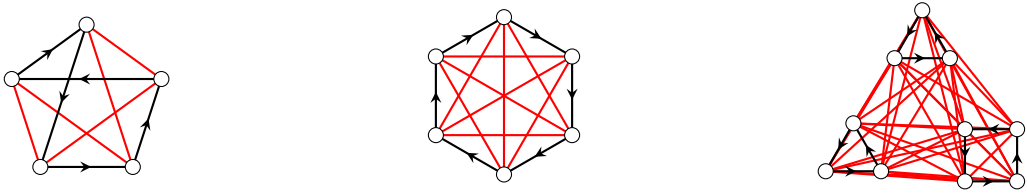
\begin{theorem}\label{thm: multiple odd cycles competitive}
For any $n\geq 4$, every ecology consisting of an odd number $k$ of disjoint obligate prey-predator cycles with inner competition is impossible.
\end{theorem}

\begin{proof}
Assume by contradiction that the ecology is possible, and write
${\xstar}=\bB^{-1}\ba$ for the feasible and stable equilibrium, so that
$\bB{\xstar}=\ba$. Let $\widehat{\bB}\coloneqq \bB\diag({\xstar})$. Since $\diag({\xstar})$ has positive determinant, $\det(\widehat{\bB})$ and $\det(\bB)$ have the same sign. It is immediate to see that the sum of the entries of the $i$-th row of $\widehat{\bB}$ is $a_i$, and hence, negative. For every $i\in[n]$, let $i^+$ denote the unique prey of species $i$ in its cycle, and write $\widehat b_{i,i^+}=-\beta_i$ with $\beta_i>0$. All the other entries in the $i$-th row of $\widehat{\bB}$ are positive. Together with the previous observation, this gives $\beta_i>\sum_{j\neq i^+}\widehat b_{ij}$ for every $i\in[n]$. Let $\bP$ be the permutation matrix associated with the mapping $\sigma(i)=i^+$. Multiplying $\widehat{\bB}$ by $\bP$ on the right moves the negative entries $-\beta_i$ onto the main diagonal. Because the permutation $\sigma$ decomposes into $k$ disjoint cycles of lengths $L_1, \dots, L_k$ (where $\sum_{r=1}^k L_r = n$), the determinant of $\bP$ is given by the product of the determinants of these cycles:
\[
    \det(\bP) = \prod_{r=1}^k (-1)^{L_r - 1} = (-1)^{\sum_{r=1}^k (L_r - 1)} = (-1)^{n-k}.
\]
Set $\bM\coloneqq -\widehat{\bB}\bP$. Then $\bM$ has positive diagonal entries $m_{ii}=\beta_i$ and non-positive off-diagonal entries, i.e.\ it is a $Z$-matrix. Moreover, the matrix $\bM$ is strictly diagonally dominated, since $m_{ii}=\beta_i>\sum_{j\neq i^+}\widehat b_{ij}=\sum_{j\neq i}|m_{ij}|$. Thus, the eigenvalues of $\bM$ have strictly positive real part, so $\bM$ is a non-singular $M$-matrix. By \cref{prop:characterizations-M-matrix}, $\det(\bM)>0$, and therefore, $\sign(\det(\widehat{\bB}\bP))=(-1)^n$. 

Combining the considerations above with the straightforward observation that $\det(\bP)^2=1$ (since $\bP$ is an orthogonal permutation matrix), yields
\begin{align*}
      \sign(\det(\bB))
    =
    \sign(\det(\widehat{\bB}))
    & =
    \sign(\det(\widehat{\bB}\bP))\sign(\det(\bP)) \\
    & =
    (-1)^n(-1)^{n-k}
    =
    (-1)^{2n-k}
    =
    (-1)^k.  
\end{align*}
Since $k$ is odd, we conclude that $\det(\bB)<0$, contradicting the stability assumption and proving that the ecology is impossible as $\det\Jx=(-1)^n\det(\diag({\xstar}))\det(\bB)$ implies that $\det(\bB)>0$.
\end{proof} 

\paragraph{Cycle networks with a competing out-group.} Let $n \geq 4$ and $k \geq 1$. We say that an ecological network consists of \emph{a single species growing by itself and competing against everyone else}, alongside a collection of $k$ disjoint obligate prey-predator cycles with inner competition if, after a relabeling of the species, species $1$ satisfies $a_1>0$ and $b_{11}>0$, while for all other species $i \in \{2, \dots, n\}$ we have $a_i<0$ and $b_{ii}>0$, and the following conditions hold:
\begin{itemize}
    \item species $1$ competes with all other species, namely $b_{1j}>0$ and $b_{j1}>0$ for every $j \in \{2, \dots, n\}$;
    \item the remaining species $\{2, \dots, n\}$ can be partitioned into $k$ disjoint directed cycles $C_1, \dots, C_k$, each of length at least $3$, such that for every $i \in \{2, \dots, n\}$, if $i^+$ denotes the unique prey of $i$ in its directed cycle, then $b_{i,i^+}<0$ and $b_{i^+,i}>0$;
    \item all other pairs of species compete, namely $b_{ij}>0$ and $b_{ji}>0$ whenever $i,j\in\{2,\dots,n\}$, $j\notin\{i,i^+\}$, and $i\neq j^+$.
\end{itemize}



\begin{figure}[htbp]
\centering

\tikzset{
  emptynode/.style={circle, draw=black, fill=white, inner sep=2pt},
  blacknode/.style={circle, draw=black, fill=black, inner sep=2pt},
  redge/.style={red, line width=0.8pt},
  bluedge/.style={blue, line width=1.2pt},
  cyclearrow/.style={
    thick,
    postaction={
      decorate,
      decoration={
        markings,
        mark=at position 0.55 with {\arrow{>}}
      }
    }
  },
  >=stealth
}

\begin{minipage}{0.26\textwidth}
\centering
\begin{tikzpicture}[scale=0.65]
  \coordinate (V1) at (90:1.55);
  \coordinate (V2) at (162:1.55);
  \coordinate (V3) at (234:1.55);
  \coordinate (V4) at (306:1.55);
  \coordinate (V5) at (18:1.55);

  \foreach \i in {2,...,5} {
    \draw[redge] (V1) -- (V\i);
  }

  \draw[redge] (V2) -- (V4);
  \draw[redge] (V3) -- (V5);

  \draw[cyclearrow] (V2) -- (V3);
  \draw[cyclearrow] (V3) -- (V4);
  \draw[cyclearrow] (V4) -- (V5);
  \draw[cyclearrow] (V5) -- (V2);

  \node[blacknode] at (V1) {};
  \foreach \i in {2,...,5} {
    \node[emptynode] at (V\i) {};
  }
\end{tikzpicture}
\end{minipage}\hfill
\begin{minipage}{0.28\textwidth}
\centering
\begin{tikzpicture}[scale=0.65]
  \coordinate (V1) at (0, 3.0);

  \coordinate (V2) at (90:1.55);
  \coordinate (V3) at (162:1.55);
  \coordinate (V4) at (234:1.55);
  \coordinate (V5) at (306:1.55);
  \coordinate (V6) at (18:1.55);

  \draw[redge] (V2) -- (V4);
  \draw[redge] (V2) -- (V5);
  \draw[redge] (V3) -- (V5);
  \draw[redge] (V3) -- (V6);
  \draw[redge] (V4) -- (V6);

  \foreach \i in {2,...,6} {
    \draw[cyclearrow] (V1) -- (V\i);
  }

  \draw[cyclearrow] (V2) -- (V3);
  \draw[cyclearrow] (V3) -- (V4);
  \draw[cyclearrow] (V4) -- (V5);
  \draw[cyclearrow] (V5) -- (V6);
  \draw[cyclearrow] (V6) -- (V2);

  \node[blacknode] at (V1) {};
  \foreach \i in {2,...,6} {
    \node[emptynode] at (V\i) {};
  }
\end{tikzpicture}
\end{minipage}\hfill
\begin{minipage}{0.44\textwidth}
\centering
\begin{tikzpicture}[scale=0.65]
  \coordinate (V1) at (-0.8, 2.8);
  \coordinate (V2) at (0.8, 2.8);

  \path (-2.0, 0) + (90:1.1) coordinate (V3);
  \path (-2.0, 0) + (210:1.1) coordinate (V4);
  \path (-2.0, 0) + (330:1.1) coordinate (V5);

  \path (2.0, 0) + (90:1.1) coordinate (V6);
  \path (2.0, 0) + (210:1.1) coordinate (V7);
  \path (2.0, 0) + (330:1.1) coordinate (V8);

  \foreach \i in {1,2} {
    \foreach \j in {3,...,8} {
      \draw[redge] (V\i) -- (V\j);
    }
  }

  \foreach \i in {3,4,5} {
    \foreach \j in {6,7,8} {
      \draw[redge] (V\i) -- (V\j);
    }
  }

  \draw[bluedge] (V1) -- (V2);

  \draw[cyclearrow] (V3) -- (V4); \draw[cyclearrow] (V4) -- (V5); \draw[cyclearrow] (V5) -- (V3);
  \draw[cyclearrow] (V6) -- (V7); \draw[cyclearrow] (V7) -- (V8); \draw[cyclearrow] (V8) -- (V6);

  \node[emptynode] at (V1) {};
  \node[emptynode] at (V2) {};
  \foreach \i in {3,...,8} {
    \node[emptynode] at (V\i) {};
  }
\end{tikzpicture}
\end{minipage}

\caption{Variations of ecological networks with isolated out-groups (black nodes) interacting with prey-predator cycles (white nodes). \textbf{Left}: A single species ($n=5$) competes against a 4-cycle. \textbf{Center}: A super-predator ($n=6$) consumes all members of a 5-cycle. \textbf{Right}: A pair in obligate mutualism ($n=8$, blue edge) competes against two disjoint 3-cycles. In all cases, cross-cycle species and non-consecutive species within cycles compete (red edges).}
\end{figure}
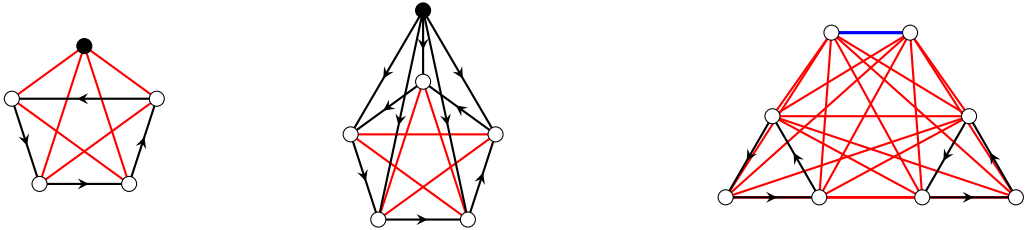

\begin{theorem}\label{THM:SingSpecCompetingOddCycles}
For any $n\geq 4$, every ecology consisting of a single species growing by itself and competing against everyone else, alongside an odd number $k$ of disjoint obligate prey-predator cycles with inner competition, is impossible.
\end{theorem}

\begin{proof} We use the same notation as in the proof of  \cref{thm: multiple odd cycles competitive} and arguing similarly we have that $\widehat{\mathbf{B}}\coloneqq \mathbf{B}\diag(\mathbf{x}^*)$ and $\mathbf{B}$ have determinants with the same sign and that the sum of the entries of the $i$-th row of $\widehat{\mathbf{B}}$ is $a_i$. We have $a_1>0$, while for $i \geq 2$, we have $a_i<0$. For every $i \in \{2, \dots, n\}$, let $i^+$ denote the unique prey of species $i$ in its cycle, and write $\widehat{b}_{i,i^+}=-\beta_i$ with $\beta_i>0$. All other off-diagonal entries in the $i$-th row are positive. This gives $\beta_i>\sum_{j\neq i^+}\widehat{b}_{ij}$ for every $i \geq 2$. Let $\mathbf{P}$ be the permutation matrix associated with the mapping $\sigma(1)=1$ and $\sigma(i)=i^+$ for $i \geq 2$. Multiplying $\widehat{\mathbf{B}}$ by $\mathbf{P}$ on the right moves the negative entries $-\beta_i$ onto the main diagonal of the lower $(n-1) \times (n-1)$ block. The determinant of $\mathbf{P}$ is $\det(\mathbf{P}) = (-1)^{n-1-k}$. Let $\mathbf{D} = \diag(1, -1, \dots, -1)$ be the diagonal matrix negating all rows except the first ($\det(\mathbf{D}) = (-1)^{n-1}$). Define $\mathbf{V} \coloneqq \mathbf{D} \widehat{\mathbf{B}} \mathbf{P}$. The matrix $\mathbf{V}$ is a block matrix:
\[
\mathbf{V} = \begin{pmatrix} V_{11} & \mathbf{r}^T \\ -\mathbf{s} & \mathbf{C} \end{pmatrix}
\]
where $V_{11} > 0$, $\mathbf{r} > 0$, $\mathbf{s} > 0$, and $\mathbf{C}$ is a strictly diagonally dominant $Z$-matrix. Thus, $\det(\mathbf{C})>0$ and $\mathbf{C}^{-1} \geq 0$. Using the Schur complement, we have $
\det(\mathbf{V}) = \det(\mathbf{C})(V_{11} + \mathbf{r}^T \mathbf{C}^{-1} \mathbf{s}) > 0$.
Finally, tracking the signs we have \[\sign(\det(\mathbf{B})) = \sign(\det(\widehat{\mathbf{B}}))= \sign(\det(\mathbf{D})) \sign(\det(\mathbf{V})) \sign(\det(\mathbf{P}))= (-1)^k\] and thus $\det(\mathbf{B}) < 0$, as $k$ is odd, contradicting the stability requirement $\det(\mathbf{B}) > 0$.
\end{proof} \vspace{0.1cm}

\paragraph{Cycle networks with a super-predator.} Let $n \geq 4$ and $k \geq 1$. We say that an ecological network consists of a single species growing on its own and predating on all other species, alongside a collection of $k$ disjoint obligate prey-predator cycles with inner competition if, after a relabeling of the species, species $1$ satisfies $a_1 > 0$ and $b_{11} > 0$, while for all other species $i \in \{2, \dots, n\}$ we have $a_i < 0$ and $b_{ii} > 0$, and the following conditions hold:
\begin{itemize}
    \item \textbf{Species 1 is a super-predator:} species $1$ benefits from consuming all other species, while harming them. Thus, $b_{1j} < 0$ and $b_{j1} > 0$ for every $j \in \{2, \dots, n\}$;
    \item \textbf{Prey-predator cycles:} the remaining species $\{2, \dots, n\}$ can be partitioned into $k$ disjoint directed cycles $C_1, \dots, C_k$, each of length at least $3$, such that for every $i \in \{2, \dots, n\}$, if $i^+$ denotes the unique prey of $i$ in its directed cycle, then $b_{i,i^+}<0$ and $b_{i^+,i}>0$;
    \item \textbf{Other pairs compete:} all other pairs of species compete, namely $b_{ij}>0$ and $b_{ji}>0$ whenever $i,j\in\{2,\dots,n\}$, $j\notin\{i,i^+\}$, and $i\neq j^+$.
\end{itemize}

\begin{theorem}\label{Thm:OnePredatingCycle}
For any $n \geq 4$, every ecology consisting of a single super-predator growing by itself, alongside an odd number $k$ of disjoint obligate prey-predator cycles with inner competition, is impossible.
\end{theorem}

\begin{proof}
We use the same notation as in the proofs of  \cref{thm: multiple odd cycles competitive} and \cref{THM:SingSpecCompetingOddCycles} arguing similarly we have that $\widehat{\mathbf{B}}\coloneqq \mathbf{B}\diag(\mathbf{x}^*)$ and $\mathbf{B}$ have determinants with the same sign and that the sum of the entries of the $i$-th row of $\widehat{\mathbf{B}}$ is $a_i$. We have $a_1>0$, while for $i \geq 2$, we have $a_i<0$. As in the proof of \cref{THM:SingSpecCompetingOddCycles}, for every $i \in \{2, \dots, n\}$, let $i^+$ denote the unique prey of species $i$ in its cycle, and write $\widehat{b}_{i,i^+}=-\beta_i$ with $\beta_i>0$. Since species $1$ harms $i$, we have $\widehat{b}_{i1} > 0$. All other off-diagonal entries in row $i$ are positive which implies $
\beta_i > \sum_{j \neq i^+} \widehat{b}_{ij}$ for every  $i \geq 2$. Let $\mathbf{P}$ be the permutation matrix associated with the mapping $\sigma(1) = 1$ and $\sigma(i) = i^+$ for $i \geq 2$ and let $\mathbf{D} = \operatorname{diag}(1, -1, \dots, -1)$. Define $\mathbf{V} \coloneqq \mathbf{D} \widehat{\mathbf{B}} \mathbf{P}$. We partition $\mathbf{V}$ as:
\[
\mathbf{V} = \begin{pmatrix} V_{11} & \mathbf{r}^T \\ -\mathbf{s} & \mathbf{C} \end{pmatrix}
\]
where $V_{11} = \widehat{b}_{11} > 0$, and $\mathbf{C}$ is a strictly diagonally dominant $Z$-matrix (and thus a non-singular M-matrix, meaning $\det(\mathbf{C}) > 0$ and $\mathbf{C}^{-1} \geq 0$). Because species $1$ is a predator, we have:
\begin{itemize}
    \item $\mathbf{s} > 0$ (since $V_{i1} = -\widehat{b}_{i1} < 0$ for $i \geq 2$).
    \item $\mathbf{r} < 0$ (since $V_{1j} = \widehat{b}_{1, \sigma(j)} < 0$ for $j \geq 2$).
\end{itemize}
To evaluate the sign of the Schur complement term $\det(\mathbf{V}) = \det(\mathbf{C})(V_{11} + \mathbf{r}^T \mathbf{C}^{-1} \mathbf{s})$, we use the row sum equations. Let $\mathbf{e} = (1, \dots, 1)^T \in \mathbb{R}^{n-1}$. Because $\mathbf{P}$ merely permutes columns, the row sums of $\mathbf{V}$ correspond directly to the row sums of $\mathbf{D}\widehat{\mathbf{B}}$:
\begin{enumerate}
    \item For row 1, $V_{11} + \mathbf{r}^T \mathbf{e} = a_1 > 0$
    \item For rows $i \geq 2$:$
    -\mathbf{s} + \mathbf{C}\mathbf{e} = -\mathbf{a}_-$
 where $\mathbf{a}_- = (a_2, \dots, a_n)^T < 0$ is the $(n-1)$-dimensional subvector of $\mathbf{a}$.
\end{enumerate}
Since $\mathbf{C}$ is an M-matrix, we multiply by $\mathbf{C}^{-1} \geq 0$: we get $\mathbf{C}^{-1}\mathbf{s} = \mathbf{e} + \mathbf{C}^{-1}\mathbf{a}_-$. Because $\mathbf{a}_- < 0$ and $\mathbf{C}^{-1} \geq 0$, the vector $\mathbf{C}^{-1}\mathbf{a}_-$ is strictly negative. Therefore, we obtain the inequality $\mathbf{C}^{-1}\mathbf{s} \leq \mathbf{e}$.
Because our predator vector satisfies $\mathbf{r} < 0$, multiplying both sides of this inequality by $\mathbf{r}^T$ reverses the inequality direction: $\mathbf{r}^T \mathbf{C}^{-1}\mathbf{s} \geq \mathbf{r}^T \mathbf{e}$. Using this result, we can bound the Schur complement term:
\[
V_{11} + \mathbf{r}^T \mathbf{C}^{-1} \mathbf{s} \geq V_{11} + \mathbf{r}^T \mathbf{e} = a_1 > 0
\]

\noindent Since $\det(\mathbf{C}) > 0$ and $V_{11} + \mathbf{r}^T \mathbf{C}^{-1} \mathbf{s} > 0$, we have $\det(\mathbf{V}) > 0$. Tracking the signs, we obtain $\operatorname{sign}(\det(\mathbf{B}))=(-1)^k$. Since $k$ is odd, we get a contradiction with the stability requirement $\det(\mathbf{B}) > 0$.
\end{proof} \vspace{0.1cm}

\paragraph{Cycle networks with an obligate mutualistic pair.} Let $n \geq 4$ and $k \geq 1$. We say that an ecological network consists of a pair of species in obligate mutualism competing against a collection of $k$ disjoint obligate prey-predator cycles with inner competition if, after a relabeling of the species, species $1$ and $2$ satisfy $a_i < 0, b_{ii} > 0$, and $b_{12}, b_{21} < 0$, while for all other species $i \in \{3, \dots, n\}$ we have $a_i<0$ and $b_{ii}>0$, and the following conditions hold:
\begin{itemize}
    \item species $1$ and $2$ compete with all other species, namely $b_{ij}>0$ and $b_{ji}>0$ for every $i \in \{1, 2\}$ and $j \in \{3, \dots, n\}$;
    \item the remaining species $\{3, \dots, n\}$ can be partitioned into $k$ disjoint directed cycles $C_1, \dots, C_k$, each of length at least $3$, such that for every $i \in \{3, \dots, n\}$, if $i^+$ denotes the unique prey of $i$ in its directed cycle, then $b_{i,i^+}<0$ and $b_{i^+,i}>0$;
    \item all other pairs of species compete, namely $b_{ij}>0$ and $b_{ji}>0$ whenever $i,j\in\{3,\dots,n\}$, $j\notin\{i,i^+\}$, and $i\neq j^+$.
\end{itemize}

\begin{theorem}
For any $n\geq 4$, every ecology consisting of a pair of species in obligate mutualism competing against an even number $k$ of disjoint obligate prey-predator cycles with inner competition is impossible.
\end{theorem}

\begin{proof}
We use the same notation as in the proofs of  \cref{thm: multiple odd cycles competitive} and \cref{THM:SingSpecCompetingOddCycles} arguing similarly we have that $\widehat{\mathbf{B}}\coloneqq \mathbf{B}\diag(\mathbf{x}^*)$ and $\mathbf{B}$ have determinants with the same sign and that the sum of the entries of the $i$-th row of $\widehat{\mathbf{B}}$ is $a_i$. Because $a_i < 0$ for all species, the sum of the entries in every $i$-th row of $\widehat{\mathbf{B}}$ is strictly negative. Let $\mathbf{P}$ be the permutation matrix associated with the mapping $\sigma(1)=1$, $\sigma(2)=2$, and $\sigma(i)=i^+$ for $i \geq 3$ and observe $\det(\mathbf{P}) = (-1)^{n-2-k}$. Multiplying $\widehat{\mathbf{B}}$ by $\mathbf{P}$ on the right moves the negative cycle entries onto the main diagonal of the lower $(n-2) \times (n-2)$ block. Let $\mathbf{D} = \operatorname{diag}(1, 1, -1, \dots, -1)$ be the diagonal matrix negating all rows except the first two ($\det(\mathbf{D}) = (-1)^{n-2}$). Define $\mathbf{V} \coloneqq \mathbf{D} \widehat{\mathbf{B}} \mathbf{P}$. The matrix $\mathbf{V}$ is partitioned as:
\[
\mathbf{V = \begin{pmatrix} \mathbf{V}_{11} & \mathbf{V}_{12} \\ \mathbf{V}_{21} & \mathbf{V}_{22} \end{pmatrix}}
\]
where $\mathbf{V}_{11}$ is a $2 \times 2$ $Z$-matrix, $\mathbf{V}_{12} > 0$, $\mathbf{V}_{21} < 0$, and $\mathbf{V}_{22}$ is an $(n-2) \times (n-2)$ strictly diagonally dominant $Z$-matrix. Consequently, $\mathbf{V}_{22}$ is a non-singular $M$-matrix, meaning $\det(\mathbf{V}_{22}) > 0$ and $\mathbf{V}_{22}^{-1} \geq 0$. Applying the Schur complement to evaluate the determinant of $\mathbf{V}$, we obtain $\det(\mathbf{V}) = \det(\mathbf{V}_{22})\det(\mathbf{S})$, where $\mathbf{S} = \mathbf{V}_{11} - \mathbf{V}_{12} \mathbf{V}_{22}^{-1} \mathbf{V}_{21}$.

We now determine the sign of $\det(\mathbf{S})$. Evaluating the row sums of $\mathbf{V}$, we have $\mathbf{V}_{11} \mathbf{1}_2 + \mathbf{V}_{12} \mathbf{1}_{n-2} = \mathbf{a}_{S_1}$ and $\mathbf{V}_{21} \mathbf{1}_2 + \mathbf{V}_{22} \mathbf{1}_{n-2} = -\mathbf{a}_{S_2}$. Solving for $\mathbf{1}_{n-2}$ and substituting yields:
\[
\mathbf{S \mathbf{1}_2 = \mathbf{a}_{S_1} + \mathbf{V}_{12} \mathbf{V}_{22}^{-1} \mathbf{a}_{S_2}}.
\]
Since $\mathbf{a}_{S_1} < 0, \mathbf{a}_{S_2} < 0, \mathbf{V}_{12} > 0$, and $\mathbf{V}_{22}^{-1} \geq 0$, we conclude that $\mathbf{S \mathbf{1}_2 < 0}$. Let $s_{ij}$ denote the entries of $\mathbf{S}$. The diagonal entries of $\mathbf{V}_{11}$ are strictly positive ($\widehat{b}_{11}, \widehat{b}_{22} > 0$), and the matrix $-\mathbf{V}_{12}\mathbf{V}_{22}^{-1}\mathbf{V}_{21}$ is non-negative. Therefore, $s_{11} > 0$ and $s_{22} > 0$. Because $\mathbf{S \mathbf{1}_2 < 0}$, we must have $s_{11} + s_{12} < 0$ and $s_{21} + s_{22} < 0$. This implies $s_{12} < -s_{11} < 0$ and $s_{21} < -s_{22} < 0$. Evaluating the determinant, we find $\det(\mathbf{S}) = s_{11}s_{22} - s_{12}s_{21} < 0$.
Because $|s_{12}| > s_{11}$ and $|s_{21}| > s_{22}$, the product of the negative off-diagonals strictly dominates the product of the positive diagonals. Since $\det(\mathbf{V}_{22}) > 0$ and $\det(\mathbf{S}) < 0$, we have $\det(\mathbf{V}) < 0$. Finally, we track the signs back to the original matrix $\operatorname{sign}(\det(\mathbf{B}))=(-1)^{k+1}$ and since $k$ is even, we obtain $\det(\mathbf{B}) < 0$ contradicting stability and proving that the ecology is impossible.
\end{proof} 

\subsection{Impossibility of Four-species Ecologies}\label{sec: 3.3}

We now complete the four-species classification. The general results of the previous subsections establish the impossibility of eight of the eleven candidate ecologies in \cref{fig:symmetricimpossible,fig: asymmetricimpossible}. It remains to treat the three asymmetric ecologies (a), (b), and (c) in \cref{fig: asymmetricimpossible}.

\begin{lemma}\label{lem: eco1asymm}
    Ecology $(a)$ in \cref{fig: asymmetricimpossible} is impossible.
\end{lemma}
\begin{proof}
Assume by contradiction that the ecology is possible, and write ${\xstar}=(x_1,x_2,x_3.x_4)^T$ for the feasible and stable equilibrium, so that $\bB{\xstar}=\ba$. We have the following sign pattern:
\[
\bB=
\begin{pmatrix}
 b_{11} & b_{12} & b_{13} & b_{14}\\
-b_{21} &  b_{22} & -b_{23} & -b_{24}\\
b_{31} & b_{32} & b_{33} & -b_{34}\\
b_{41} & b_{42} & -b_{43} & b_{44}
\end{pmatrix},
\quad
\ba=
\begin{pmatrix}
a_1\\ a_2\\ -a_3\\ -a_4
\end{pmatrix}.
\]
We eliminate the second species, that is we take $P=\{2\}$ and $Q=\{1,3,4\}$. We immediately have $\det(\bB_{PP})=b_{22}>0$. Consider the reduced system $\bC\bx_Q=\bA_Q$, where
\[
\bC=
\begin{pmatrix}
c_{11} & c_{13} & c_{14}\\
c_{31} & c_{33} & -c_{34}\\
c_{41} & -c_{43} & c_{44}
\end{pmatrix}, \quad
\bA_Q=\begin{pmatrix}
?\\-A_3\\-A_4
\end{pmatrix}.
\]
The choice of signs above is not arbitrary. We do not know (and do not need) the sign of $(\bA_Q)_1$. However, $(\bA_Q)_2$ and $(\bA_Q)_3$ are negative since, for instance, $(\bA_Q)_2=-a_3-a_2b_{32}b_{22}^{-1}$. Direct computations also show that all $(\bC)_{ij}$ have known signs, apart from $(\bC)_{23}$ and $(\bC)_{32}$, whose signs are not known a priori. However, assuming feasibility and writing explicitly the second and third equation of the reduced system, we get $(\bC)_{23}x_4=-A_3-c_{31}x_1-c_{33}x_3<0$ and $(\bC)_{32}x_3=-A_4-c_{41}x_1-c_{44}x_4<0$, justifying the chosen sign representation. Having fixed scalars and signs, from the last two equations of the reduced system we now read $c_{33}x_3-c_{34}x_4<0$ and $-c_{43}x_3+c_{44}x_4<0$. Combining these two, gives the key inequality $c_{33}c_{44}-c_{34}c_{43}<0$.

We now expand the determinant of $\bC$ along the first row, and obtain
\[
\det(\bC)
=
c_{11}(c_{33}c_{44}-c_{34}c_{43})
-c_{13}(c_{31}c_{44}+c_{34}c_{41})
-c_{14}(c_{31}c_{43}+c_{33}c_{41}).
\]
Our key inequality forces the first summand to be negative, and since the other summands have non-ambiguous signs, we conclude that $\det(\bC)<0$. By \cref{lem: schur}, $\det(\bB)=b_{22}\det(\bC)<0$, contradicting the stability assumption, proving that the ecology is impossible.
\end{proof} \vspace{0.1cm}

\noindent We now turn our attention to ecology $(b)$ in \cref{fig: asymmetricimpossible}.
\begin{lemma}\label{lem: eco2asymm}
    Ecology $(b)$ in \cref{fig: asymmetricimpossible} is impossible.
\end{lemma}
\begin{proof}
Assume by contradiction that the ecology is possible, and write ${\xstar}=(x_1,x_2,x_3,x_4)^T$ for the feasible and stable equilibrium, so that $\bB{\xstar}=\ba$. We have the following sign pattern:
\begin{align*}
\bB=
\begin{pmatrix}
 b_{11} & -b_{12} & -b_{13} & -b_{14}\\
 -b_{21} & b_{22} & -b_{23} & -b_{24}\\
 b_{31} & b_{32} & b_{33} & -b_{34}\\
 b_{41} & b_{42} & -b_{43} & b_{44}
\end{pmatrix},
\quad
\ba=
\begin{pmatrix}
a_1\\ a_2\\ -a_3\\ -a_4
\end{pmatrix}.
\end{align*}
We will eliminate the last two species, hence we take $P=\{3,4\}$ and $Q=\{1,2\}$. 

We first prove that $\det(\bB_{PP})<0$. The last two equilibrium equations give $\bB_{PQ}\bx_Q+\bB_{PP}\bx_P=\ba_P$. Since $\bB_{PQ}$ and $\bx_Q$ have positive entries, we have $\bB_{PQ}\bx_Q>0$. Since $\ba_P<0$, this gives
$\bB_{PP}\bx_P=\ba_P-\bB_{PQ}\bx_Q<0$. Written componentwise, this gives $b_{33}x_3-b_{34}x_4<0$ and $-b_{43}x_3+b_{44}x_4<0$, which implies $\det(\bB_{PP})<0$. Consider the reduced system $\bC\bx_Q=\bA_Q$. Since
\[
    \bB_{PP}^{-1}
    =
    \frac{1}{\det(\bB_{PP})}
    \begin{pmatrix}
        b_{44} & b_{34}\\
        b_{43} & b_{33}
    \end{pmatrix}
\]
and $\det(\bB_{PP})<0$, every entry of $\bB_{PP}^{-1}$ is negative. Since $\ba_P<0$, we get $\bB_{PP}^{-1}\ba_P>0$. Since $\bB_{QP}$ has negative entries, this gives $\bB_{QP}\bB_{PP}^{-1}\ba_P<0$, and therefore $\bA_Q>0$.

For the matrix $\bC$, notice that $\bB_{QP}\bB_{PP}^{-1}\bB_{PQ}$ has positive entries, because $\bB_{QP}<0$, $\bB_{PP}^{-1}<0$, and $\bB_{PQ}>0$. Hence the off-diagonal entries of $\bC=\bB_{QQ}-\bB_{QP}\bB_{PP}^{-1}\bB_{PQ}$ are negative. We may write
\[
    \bC=
    \begin{pmatrix}
        c_{11} & -c_{12}\\
        -c_{21} & c_{22}
    \end{pmatrix},
    \quad
    \bA_Q=
    \begin{pmatrix}
        A_1\\ A_2
    \end{pmatrix}.
\]
Notice that the signs of the diagonal coefficients is not known a priori. However, from the equations of the reduced system we may read $(\bC)_{11}x_1=c_{12}x_2+A_1>0$ and $(\bC)_{22}x_2=c_{21}x_1+A_2>0$, justifying the chosen sign representation. Having fixed scalars and signs, we now read $c_{11}x_1>c_{12}x_2$, and $c_{22}x_2>c_{21}x_1$, which imply $\det(\bC)=c_{11}c_{22}-c_{12}c_{21}>0$. By \cref{lem: schur}, $\det(\bB)=\det(\bB_{PP})\det(\bC)$ and since $\det\bB_{PP}<0$, it follows that $\det(\bB)<0$. This contradicts the stability assumption, proving that the ecology is impossible.
\end{proof} \vspace{0.1cm}

\noindent Lastly, we prove that also ecology $(c)$ in \cref{fig: asymmetricimpossible} is impossible.

\begin{lemma}\label{lem: eco3asymm}
    Ecology $(c)$ in \cref{fig: asymmetricimpossible} is impossible.
\end{lemma}
\begin{proof}
Assume by contradiction that the ecology is possible, and write ${\xstar}(x_1,x_2,x_3,x_4)^T$ for the feasible and stable equilibrium, so that $\bB{\xstar}=\ba$. We have the following sign pattern:
\begin{align*}
\bB=
\begin{pmatrix}
 b_{11} & b_{12} & b_{13} & b_{14}\\
 b_{21} & b_{22} & -b_{23} & b_{24}\\
 b_{31} & -b_{32} & b_{33} & b_{34}\\
 -b_{41} & b_{42} & b_{43} & b_{44}
\end{pmatrix},
\quad
\ba=
\begin{pmatrix}
a_1\\ -a_2\\ -a_3\\ -a_4
\end{pmatrix}.
\end{align*}
We eliminate the second and third species, that is we take $P=\{2,3\}$ and $Q=\{1,4\}$. Arguing as in the proof of \cref{lem: eco2asymm}, we find $\det(\bB_{PP})<0$, which implies ${\bB_{PP}^{-1}}<0$.

Consider the reduced system $\bC\bx_Q=\bA_Q$. Since $\bB_{QP}$ and $\bB_{PQ}$ have positive entries and $\bB_{PP}^{-1}<0$, the matrix $\bB_{QP}\bB_{PP}^{-1}\bB_{PQ}$ has negative entries. Hence all the entries of $\bC=\bB_{QQ}-\bB_{QP}\bB_{PP}^{-1}\bB_{PQ}$ are positive, apart from $(\bC)_{21}$, whose sign is not known a priori. Moreover, since $\ba_P<0$, we have $\bB_{PP}^{-1}\ba_P>0$, and therefore the second entry of $\bA_Q=\ba_Q-\bB_{QP}\bB_{PP}^{-1}\ba_P$ is negative. We may thus write
\[
\bC=
\begin{pmatrix}
c_{11} & c_{14}\\
-c_{41} & c_{44}
\end{pmatrix},
\qquad
\bA_Q=
\begin{pmatrix}
?\\
-A_4
\end{pmatrix},
\]
where the negative sign of $(\bC)_{21}$ is forced by feasibility, since the second equation of the reduced system gives $(\bC)_{21} x_1=-c_{44}x_4-A_4<0$. It follows that $\det(\bC)=c_{11}c_{44}+c_{14}c_{41}>0$. By \cref{lem: schur}, $\det(\bB)=\det(\bB_{PP})\det(\bC)<0$, contradicting the stability assumption and proving that the ecology is impossible.
\end{proof} \vspace{0.1cm}

\noindent We are now ready to state the following theorem. 

\begin{theorem}
All the 11 four-species ecologies represented in \cref{fig:symmetricimpossible}
 and \cref{fig: asymmetricimpossible} are impossible.
 \end{theorem}

 \begin{proof}
     This follows directly from
\cref{thm: obligate mutualism,thm: competing mutualism,thm: multiple odd cycles competitive,THM:SingSpecCompetingOddCycles,Thm:OnePredatingCycle,lem: eco1asymm,lem: eco2asymm,lem: eco3asymm},
which together cover all five symmetric and all six asymmetric four-species ecologies.
 \end{proof} \vspace{0.1cm}

We close this section with a remark on how the Schur-complement constraints obtained here may be incorporated into the computational Grassmannian framework developed in \cite{Celik2025}, which can be analyzed further in future work. In \cite{Celik2025}, the feasibility and stability conditions for the Lotka--Volterra system were encoded through the matrix $[\diag(\ba)\mid \bB]\in\R^{n\times 2n}$,
whose maximal minors define Pl\"ucker coordinates on the real Grassmannian
$\Gr_{\R}(n,2n)$. Starting from the signs prescribed by an ecological interaction pattern, the algorithm developed there searches for \emph{weak potential completions} of the resulting partial sign assignment, using the oriented Grassmann--Pl\"ucker relations together with feasibility and weak stability constraints. If no weak potential completion exists, then the corresponding ecology is impossible. The converse does not hold: surviving completions may still violate additional algebraic constraints not incorporated into the algorithm. The Schur-complement arguments developed in the present section provide such additional constraints. In several cases, they determine the sign of a Pl\"ucker coordinate whose opposite sign is shared by every weak potential completion returned by the algorithm. More precisely: for ecology~(d) in \cref{fig:symmetricimpossible}, the algorithm returns $64$ weak potential completions, all satisfying $\chi_{1278}=-$, whereas the argument in \cref{thm: competing mutualism} forces $\chi_{1278}=+$; for ecology~(e) in \cref{fig:symmetricimpossible}, the algorithm returns $64$ weak potential completions, all satisfying $\chi_{1278}=+$, whereas the same argument forces $\chi_{1278}=-$; for ecology~(a) in \cref{fig: asymmetricimpossible}, the algorithm returns $1996$ weak potential completions, all satisfying $\chi_{1678}=+$, whereas the Schur-complement argument in \cref{lem: eco1asymm} forces $\chi_{1678}=-$; for ecology~(c) in \cref{fig: asymmetricimpossible}, the algorithm returns $1528$ weak potential completions, all satisfying $\chi_{1567}=-$, whereas the argument in \cref{lem: eco3asymm} forces $\chi_{1567}=+$. Thus, in each of these examples, incorporating the sign constraints obtained from the Schur-complement reductions would eliminate all weak potential completions. This suggests a natural refinement of the Grassmannian completion algorithm by supplementing the oriented Grassmann--Pl\"ucker, feasibility, and weak stability conditions with determinant and minor-sign constraints arising from suitable Schur complements.

\section{Higher-{o}rder Interaction Networks and Symmetric Games}\label{sec:higher}

The impossibility results above are proved within the generalized Lotka--Volterra model, which describes pairwise interactions. This leaves a natural question: can an impossible ecology become possible after adding higher-order terms \cite{gibbs2024can}? We first illustrate this phenomenon with the case of obligate mutualism.

\begin{example}\label{ex: higher obligate mutualism}
Consider four species subject to the dynamics
\begin{equation*}
    \dot x_i=
    x_i\Big(
    -1-x_i+2\sum_{j\neq i}^{4}x_j
    -\frac{2}{3}\sum_{j<k}x_jx_k\Big),
    \quad i \in [4].
\end{equation*}
The constant and linear terms in the parentheses correspond to an obligate mutualism sign pattern: the intrinsic growth rates are negative, the self-interactions are competitive, and all pairwise interactions between distinct species are mutualistic. Let $x_\star=(1,1,1,1)$. This is a feasible equilibrium for the system above. Computing the Jacobian at $x_\star$ gives $J(x_\star)=-3I_4$, so the equilibrium is asymptotically stable. The obligate mutualism sign pattern, which is impossible under pairwise interactions alone, may therefore become possible after adding third-order interactions.
\end{example}

\cref{ex: higher obligate mutualism} is an instance of a higher-order Lotka--Volterra system, a class of models used to describe higher-order interactions in ecological communities (see, for example, \cite{BaireyKelsicKishony2016,GrilliEtAl2017,SinghBaruah2021}). In this section, we recall the classical correspondence between Lotka--Volterra and replicator dynamics, established for pairwise interactions in \cite{Hofbauer1998}. {Building on its recent extension to higher-order interactions in \cite{Gokhalev2}, we make this correspondence explicit and describe the symmetric multiplayer games giving rise to higher-order Lotka--Volterra systems.} We then study the persistence of totally mixed Nash equilibria under perturbations of the payoffs. Finally, we use this correspondence to translate our results on impossible ecologies into nonexistence results for evolutionarily stable strategies of two-player symmetric games.

Let $n\geq 1$ be the number of species in a population and let $d\geq 2$. Let $r=(r_1,\dots,r_n)\in\R^n$ be their growth rates. For each $q \in [d-1]$, let $S^q=(S^q_{i j_1\cdots j_q})$ be a tensor in $\R^{{n}^{q+1}}$. We interpret these tensors as multidimensional tables prescribing interactions of order $q+1$, so that the entry $S^q_{i j_1\cdots j_q}$ may be seen as the joint contribution of the (not necessarily distinct) species indexed by $j_1,\dots,j_q$ to the growth of species $i$. The $d$-th order Lotka--Volterra system is
\begin{equation}\label{eq: hoglv general}
    \dot x_i=x_i\Big(r_i+\sum_{q=1}^{d-1}\sum_{j_1,\dots,j_q=1}^{n}S^q_{i j_1\cdots j_q}x_{j_1}\cdots x_{j_q}\Big), \quad i \in [n].
\end{equation}
It describes the evolution of the $n$ species subject to the interactions prescribed by the tensors $S^q$. We may assume that these tensors are symmetric with respect to the last $q$ indices. To see this, define 
\begin{equation*}
    \operatorname{Sym}(S^q)_{i j_1\cdots j_q}\coloneqq \frac{1}{q!}\sum_{\sigma\in\mathfrak S_q}S^q_{i j_{\sigma(1)}\cdots j_{\sigma(q)}},
\end{equation*}
where $\mathfrak S_q$ denotes the symmetric group acting on $[q]$. A direct computation shows that 
\begin{equation*}
    \sum_{j_1,\dots,j_q=1}^{n} \operatorname{Sym}(S^q)_{i j_1\cdots j_q}
    x_{j_1}\cdots x_{j_q} = \sum_{j_1,\dots,j_q=1}^{n} S^q_{i j_1\cdots j_q}x_{j_1}\cdots x_{j_q},
\end{equation*}
and so the vector field of \labelcref{eq: hoglv general} remains unchanged. From now on, every $S^q$ is assumed symmetric in its last $q$ indices, so it is regarded as an element of $\operatorname{Hom}(\operatorname{Sym}^q\R^n,\R^n)$. Notice that the sign convention in \labelcref{eq: lv} is recovered by setting $r=\ba$ and $S^1=-\bB$. If we set
\begin{equation}\label{eq: Gi polynomials}
    G_i(x)\coloneqq r_i+ \sum_{q=1}^{d-1} \sum_{j_1,\dots,j_q=1}^{n} S^q_{i j_1\cdots j_q} x_{j_1}\cdots x_{j_q},
\end{equation}
\labelcref{eq: hoglv general} rewrites as $\dot x=\diag(x)G(x)$. As in the classical Lotka--Volterra system, a feasible equilibrium is a point $x_\star\in\R^n_{>0}$ such that $G(x_\star)=0$. Again, the coordinate hyperplanes and the positive orthant are invariant for the dynamics of \labelcref{eq: hoglv general}.

\Cref{ex: higher obligate mutualism} is not special to obligate mutualism. Every impossible ecology arising through the classical Lotka--Volterra model can be made possible by adding higher-degree interaction terms, as the following result shows.

\begin{theorem}\label{thm: higher stabilization}
Let $n\geq 1$ and let $r\in\R^n$ and $S^1\in\R^{n^2}$ be arbitrary parameters. Then there exist tensors $S^2\in \R^{n^3}$ and $S^3\in \R^{n^4}$ such that the $4$-th order Lotka--Volterra system 
\begin{equation*}
    \dot x_i=x_i\Big(r_i+\sum_{j=1}^{n}S^1_{ij}x_j+\sum_{j,k=1}^{n}S^2_{ijk}x_jx_k+\sum_{j,k,\ell=1}^{n}S^3_{ijk\ell}x_jx_kx_\ell
    \Big),
    \quad i \in [n],
\end{equation*}
has a feasible asymptotically stable equilibrium at $\mathbf 1\coloneq (1,\dots,1)$.
\end{theorem}

\begin{proof}
Let $T(x)=x_1+\cdots+x_n$ and consider the polynomials
\[\phi_0(x)\coloneq 3\left(\frac{T(x)}{n}\right)^2-2\left(\frac{T(x)}{n}\right)^3, \qquad \phi_j(x)\coloneq \frac{T(x)x_j}{n}-2\left(\frac{T(x)}{n}\right)^2+\left(\frac{T(x)}{n}\right)^3\]
for $j\in[n]$. These polynomials contain only terms of degrees $2$ and $3$. Moreover, $\phi_0(\mathbf 1)=1$ and $\nabla\phi_0(\mathbf 1)=0$, while $\phi_j(\mathbf 1)=0$ and $\frac{\partial\phi_j}{\partial x_k}(\mathbf 1)=\delta_{jk}$ for every $j,k\in[n]$. For each $i\in[n]$, we define the polynomials
\[q_i(x)\coloneq -\Big(r_i+\sum_{j=1}^{n}S^1_{ij}\Big)\phi_0(x) +\sum_{j=1}^{n}\bigl(-\delta_{ij}-S^1_{ij}\bigr)\phi_j(x).\]
Then, $q_i(\mathbf 1)=-r_i-\sum_{j=1}^{n}S^1_{ij}$, while $\frac{\partial q_i}{\partial x_j}(\mathbf 1)=-\delta_{ij}-S^1_{ij}$.
Since $q_i$ contains only quadratic and cubic terms, its coefficients define certain tensors $S^2\in\R^{n^3}$ and $S^3\in\R^{n^4}$ such that
\[q_i(x)=\sum_{j,k=1}^{n}S^2_{ijk}x_jx_k+\sum_{j,k,\ell=1}^{n}S^3_{ijk\ell}x_jx_kx_\ell.\]
Using these tensors, consider the higher-order Lotka--Volterra system
\[\dot x_i=F_i(x)\coloneq x_i\Big(r_i+\sum_{j=1}^{n}S^1_{ij}x_j+q_i(x)\Big), \quad i \in [n].\]
By the choice of $q_i$, we have $F_i(\mathbf 1)=r_i+\sum_{j=1}^{n}S^1_{ij}+q_i(\mathbf 1)=0$ for each $i$, so $\mathbf 1$ is a feasible equilibrium. Moreover, differentiating the vector field at $\mathbf 1$, we find 
\[\frac{\partial F_i}{\partial x_j}(\mathbf 1)=\delta_{ij}\Big(r_i+\sum_{k=1}^{n}S^1_{ik}+q_i(\mathbf 1)\Big)+S^1_{ij}+\frac{\partial q_i}{\partial x_j}(\mathbf 1)=-\delta_{ij}.\]
Hence, the Jacobian of the system computed at $\mathbf 1$ is $-I_n$, so that $\mathbf 1$ is asymptotically stable.
\end{proof} \vspace{0.1cm}
\cref{thm: higher stabilization} removes every pairwise sign obstruction after quadratic and cubic terms are added. We return to its payoff-robust game-theoretic consequence in \cref{cor: adding players}, after recalling the correspondence between higher-order Lotka--Volterra and replicator dynamics in the next section.

\subsection{Higher-{O}rder Lotka--Volterra and Replicator Dynamics}

We now turn our attention to the replicator dynamics, which describes how the frequencies of competing types change under frequency-dependent selection. Consider a population divided into $n+1$ subtypes. Its state is a point $y(t)=(y_1(t),\dots,y_{n+1}(t))\in\Delta_n$, where $y_i(t)$ is interpreted as the frequency of type $i$ in the population at time $t$, while $\Delta_n$ denotes the standard $n$-dimensional probability simplex. Assuming that the population is large and well-mixed, we can suppose that $y(t)$ is a smooth function that evolves on $\Delta_n$. We will drop the $t$ and speak of $y$ and $y_i$. Let $A=(A_{i k_1\cdots k_{d-1}})\in\operatorname{Hom}(\operatorname{Sym}^{d-1}\R^{n+1},\R^{n+1})$. We model the \emph{fitness of subtype $i$}, which can be thought of as its reproductive success, with the homogeneous polynomial of degree $d-1$
\begin{equation}\label{eq: higher fitness}
    f_i(y)=\sum_{k_1,\dots,k_{d-1}=1}^{n+1}A_{i k_1\cdots k_{d-1}}y_{k_1}\cdots y_{k_{d-1}}, \quad i \in [n+1].
\end{equation}
As with the tensor $S^q$ from \eqref{eq: hoglv general}, $A$ need not be symmetric in the last $d-1$ indices, but we may assume so without loss of generality. We call $A$ the \emph{fitness tensor}. The coefficient $A_{i k_1\cdots k_{d-1}}$ can be seen as the contribution to the fitness of subtype $i$ when its $d-1$ partners have types $k_1,\dots,k_{d-1}$. The rate $\dot y_i/y_i$ measures the evolutionary success of subtype $i$ and can be expressed as the difference between the fitness $f_i(y)$ and the average fitness of the population, $\bar f(y)=\sum_{i=1}^{n+1}y_i f_i(y)$. This gives the following system of ODEs, known as replicator equations:
\begin{equation}\label{eq: higher replicator}
    \dot y_i=y_i(f_i(y)-\bar f(y)),
    \quad i \in [n+1].
\end{equation}
The faces of $\Delta_n$ and its interior are invariant under the dynamics of \labelcref{eq: higher replicator}. If $y\in\operatorname{int}\Delta_n$, then $\dot y_i=0$ for every $i$ if and only if $f_i(y)=\bar f(y)$ for every $i$, which is equivalent to $f_1(y)=\cdots=f_{n+1}(y)$. We call such a point an \emph{internal equilibrium}.

\noindent Let $H=\{y\in\Delta_n:y_{n+1}>0\}$. We define the map
\begin{equation}\label{eq: diffeo simplex}
    \Phi:H\longrightarrow\R^n_{\geq0},
    \quad
    \Phi(y)=
    \left(\frac{y_1}{y_{n+1}},\dots,\frac{y_n}{y_{n+1}}\right).
\end{equation}
This map is a diffeomorphism with inverse $\Phi^{-1}(x)=(x_1,\dots,x_n,1)/(1+\sum_{j=1}^{n}x_j)$. Let $x_i=y_i/y_{n+1}$ for $i \in [n]$. We fix $q\in[d-1]$ and $j_1,\dots,j_q\in[n]$ and define parameters
\begin{equation}\label{eq: psi coordinates}
    \begin{split}
    r_i&\coloneq A_{i,n+1,\dots,n+1}-A_{n+1,n+1,\dots,n+1},\\
    S^q_{i j_1\cdots j_q}&\coloneq\binom{d-1}{q}(A_{i,j_1,\dots,j_q,n+1,\dots,n+1}-A_{n+1,j_1,\dots,j_q,n+1,\dots,n+1}).
    \end{split}
\end{equation}
In the right-hand side of the second expression, there are $d-1-q$ fixed indices of $A$ equal to $n+1$, and the binomial coefficient counts the possible positions that can be occupied by the $q$ non-fixed indices among the $d-1$ partner entries. We refer to $n+1$ as the dummy index and to the $j_a$ as the non-dummy ones.\\

{\noindent The classical correspondence between replicator dynamics and pairwise Lotka--Volterra systems is described in
\cite{bomze1983lotka,Hofbauer1998}. The higher-order extension is outlined in \cite{Gokhalev2}. Here we give the correspondence explicitly for arbitrary order using tensors symmetric in their partner indices, with all computations, and introduce notation that will be used throughout. Moreover, we exploit it to study the associated fitness tensors and symmetric multiplayer games in Section~\ref{sec: 4.2}.
}

\begin{theorem}\label{thm: higher hofbauer}
Let $A\in\operatorname{Hom}(\operatorname{Sym}^{d-1}\R^{n+1},\R^{n+1})$, and let $(r,S^1,\dots,S^{d-1})$ be defined by \labelcref{eq: psi coordinates}. The diffeomorphism $\Phi:H\to\R^n_{\geq0}$ maps the orbits of the replicator system \labelcref{eq: higher replicator} to the orbits of the higher-order Lotka--Volterra system \labelcref{eq: hoglv general}, after the smooth increasing time reparametrization $d\tau=y_{n+1}^{d-1}dt$. In particular, we have a bijection between feasible equilibria of \labelcref{eq: hoglv general} and internal equilibria of \labelcref{eq: higher replicator} which respects their local stability properties.
\end{theorem}

\begin{proof}
We start by differentiating $x_i=y_i/y_{n+1}$ along the replicator system and obtain
\begin{equation}\label{eq: coordinate derivative}
    \dot x_i=
    \frac{\dot y_i y_{n+1}-y_i\dot y_{n+1}}{y_{n+1}^2}=
    \frac{\dot y_i}{y_{n+1}}
    -x_i\frac{\dot y_{n+1}}{y_{n+1}}.
\end{equation}
From \labelcref{eq: higher replicator}, we get for $i \in [n]$
\begin{equation}\label{eq: replicator ratios}
    \frac{\dot y_i}{y_{n+1}}
    =x_i(f_i(y)-\bar f(y)),
    \quad
    \frac{\dot y_{n+1}}{y_{n+1}}
    =f_{n+1}(y)-\bar f(y).
\end{equation}
Plugging \labelcref{eq: replicator ratios} in \labelcref{eq: coordinate derivative} cancels the average fitness and gives for $i \in [n]$
\begin{equation}\label{eq: fitness difference}
    \dot x_i=x_i(f_i(y)-f_{n+1}(y)).
\end{equation}

\noindent Let $B_{i k_1\cdots k_{d-1}}\coloneq A_{i k_1\cdots k_{d-1}}-A_{n+1,k_1\cdots k_{d-1}}$. We have
\begin{equation*}
    f_i(y)-f_{n+1}(y)=B_{i,n+1,\dots,n+1} \, y_{n+1}^{d-1}+\sum_{q=1}^{d-1}\binom{d-1}{q}y_{n+1}^{d-1-q}\!\!\!\sum_{j_1,\dots,j_q=1}^{n}B_{i,j_1,\dots,j_q,n+1,\dots,n+1}y_{j_1}\cdots y_{j_q}.
\end{equation*}
Each summand in the term of order $q$ contains $q$ factors $y_{j_a}$ and $d-1-q$ factors $y_{n+1}$, so every term contains the common factor $y_{n+1}^{d-1}$. For $j\leq n$, the definition of $x_j$ gives $y_j=x_jy_{n+1}$ and, using \labelcref{eq: psi coordinates}, the above expression reduces to
\begin{equation*}
    f_i(y)-f_{n+1}(y)=y_{n+1}^{d-1}
    \Big(r_i+\sum_{q=1}^{d-1}\sum_{j_1,\dots,j_q=1}^{n}S^q_{i j_1\cdots j_q}x_{j_1}\cdots x_{j_q}\Big)=y_{n+1}^{d-1}G_i(x).
\end{equation*}
Equation \labelcref{eq: fitness difference} now reads $\dot x_i=x_i y_{n+1}^{d-1}G_i(x)$. Since $y_{n+1}>0$ on $H$, the function $\tau(t)=\int_0^t y_{n+1}(s)^{d-1}\,ds$ is strictly increasing. With the substitution $d\tau=y_{n+1}^{d-1}\,dt$, we get $dx_i/d\tau=x_iG_i(x)$ for every $i$, recovering the higher-order Lotka--Volterra system.
\end{proof} \vspace{0.1cm}

\noindent The proof above identifies the higher-order Lotka--Volterra coefficients obtained from the fitness tensor. We collect this relationship in the linear map
\begin{equation*}
    \begin{split}
    \Psi_{d-1}:\operatorname{Hom}(\operatorname{Sym}^{d-1}\R^{n+1},\R^{n+1})&\longrightarrow\R^n\oplus\bigoplus_{q=1}^{d-1}\operatorname{Hom}(\operatorname{Sym}^q\R^n,\R^n),\\
    A&\longmapsto(r,S^1,\dots,S^{d-1}).
    \end{split}
\end{equation*}

The map $\Psi_{d-1}$ is not injective: different fitness tensors may determine the same higher-order Lotka--Volterra system, because replicator dynamics depends only on differences of fitnesses. The next proposition identifies this ambiguity precisely. We then show that every higher-order Lotka--Volterra system arises from a fitness tensor and choose a canonical representative in each equivalence class.

\begin{proposition}\label{prop: fitness shift}
Let $C\in\operatorname{Sym}^{d-1}((\R^{n+1})^\ast)$, and define $\widetilde{A}_{i k_1\cdots k_{d-1}}\coloneq A_{i k_1\cdots k_{d-1}}+C_{k_1\cdots k_{d-1}}$ for every $i,k_1,\dots,k_{d-1}$. Then $A$ and $\widetilde{A}$ define the same replicator vector field on $\Delta_n$. In particular, $\ker\Psi_{d-1}$ is the space of tensors of this form.
\end{proposition}

\begin{proof}
Let $c(y)\coloneq \sum_{k_1,\dots,k_{d-1}=1}^{n+1}C_{k_1\cdots k_{d-1}}y_{k_1}\cdots y_{k_{d-1}}$. The fitnesses defined by $\widetilde{A}$ satisfy $g_i(y)=f_i(y)+c(y)$ for every $i$. Since $\sum_i y_i=1$ on $\Delta_n$, their average satisfies
\begin{equation*}
    \bar{g}(y)=\sum_{i=1}^{n+1}y_i(f_i(y)+c(y))=\bar f(y)+c(y).
\end{equation*}
It follows that $g_i(y)-\bar g(y)=f_i(y)-\bar f(y)$ for every $i$, and therefore the replicator vector fields coincide. Every tensor of the stated form belongs to $\ker\Psi_{d-1}$, as all differences in \labelcref{eq: psi coordinates} vanish. For the converse, suppose that $\Psi_{d-1}(A)=0$. Let $i\in[n]$ and consider a tuple $(k_1,\dots,k_{d-1})\in[n+1]^{d-1}$. Let $q$ be the number of entries of this tuple that lie in $[n]$. After permuting the last $d-1$ indices, the tuple has the form $(j_1,\dots,j_q,n+1,\dots,n+1)$. The vanishing of the corresponding coordinate of $\Psi_{d-1}(A)$ gives $A_{i k_1\cdots k_{d-1}}=A_{n+1,k_1\cdots k_{d-1}}$. Set $C_{k_1\cdots k_{d-1}}\coloneq A_{n+1,k_1\cdots k_{d-1}}$. Then $A_{i k_1\cdots k_{d-1}}=C_{k_1\cdots k_{d-1}}$ for every $i\in[n+1]$, so every row slice of $A$ is equal to $C$, which concludes the proof.
\end{proof} \vspace{0.1cm}

The kernel in \cref{prop: fitness shift} describes the freedom in choosing a fitness tensor for a fixed higher-order Lotka--Volterra system. The next result removes this freedom by choosing one representative in every class.

\begin{proposition}\label{prop: canonical representative}
The map $\Psi_{d-1}$ is surjective. Each class modulo $\ker\Psi_{d-1}$ has a unique representative $A^0$ whose $(n+1)$-st row slice is zero. For given parameters $(r,S^1,\dots,S^{d-1})$, this representative is defined by $A^0_{n+1,k_1,\dots,k_{d-1}}=0$ for all $k_1,\dots,k_{d-1}$, by $A^0_{i,n+1,\dots,n+1}=r_i$ for $i\leq n$, and, for $q \in [d-1]$, by
\begin{equation}\label{eq: canonical representative}
    A^0_{i,j_1,\dots,j_q,n+1,\dots,n+1}
    =
    \frac{S^q_{i j_1\cdots j_q}}{\binom{d-1}{q}},
    \quad i,j_1,\dots,j_q\in[n],
\end{equation}
with the same value for every permutation of the last $d-1$ indices.
\end{proposition}

\begin{proof}
The last row of $A^0$ is zero, so the fitness of subtype $n+1$ is $f_{n+1}(y)=0$. For $i\leq n$, the entries with exactly $q$ non-dummy indices occur in $\binom{d-1}{q}$ positions. From \labelcref{eq: canonical representative} we get
\begin{equation*}
    f_i(y)=r_i y_{n+1}^{d-1}+\sum_{q=1}^{d-1}y_{n+1}^{d-1-q}\sum_{j_1,\dots,j_q=1}^{n}S^q_{i j_1\cdots j_q}y_{j_1}\cdots y_{j_q}=y_{n+1}^{d-1}G_i(x).
\end{equation*}
Reading the coefficients through \labelcref{eq: psi coordinates} gives $\Psi_{d-1}(A^0)=(r,S^1,\dots,S^{d-1})$, which proves surjectivity. For any tensor $A$, we subtract from every row slice the tensor $C$ defined by $C_{k_1\cdots k_{d-1}}=A_{n+1,k_1\cdots k_{d-1}}$. By \cref{prop: fitness shift}, this does not change the class modulo $\ker\Psi_{d-1}$, and the resulting tensor has zero last row slice. If two representatives in the same class both have zero last row slice, their difference lies in the kernel and has the form described in \cref{prop: fitness shift}. Its last row slice is $C$, so $C=0$ and the two representatives coincide.
\end{proof} 

\begin{definition}\label{def: canonical representative}
    We call the unique representative of a class modulo $\ker\Psi_{d-1}$ whose $(n+1)$-st row slice is zero the \emph{canonical representative} of the class.
\end{definition}

\noindent Below, we explicitly describe the correspondence and the construction of the canonical representative for $d=2$ and $d=3$.

\begin{example}
For $d=2$, the fitness tensor $A$ is a matrix and \labelcref{eq: hoglv general} reduces to $\dot x_i=x_i(r_i+\sum_{j=1}^nS^1_{ij}x_j)$, i.e., to the usual Lotka--Volterra system. Formula \labelcref{eq: psi coordinates} gives $r_i=A_{i,n+1}-A_{n+1,n+1}$ and $S^1_{ij}=A_{ij}-A_{n+1,j}$. The canonical representative is the matrix whose last row is zero, whose last column above that row is $r$, and whose upper-left $n\times n$ block is $S^1$. For $d=3$, the fitness tensor $A$ is a third-order tensor and \labelcref{eq: hoglv general} reduces to
\begin{equation*}\dot x_i=x_i\Big(r_i+\sum_{j=1}^{n}S^1_{ij}x_j+\sum_{j,k=1}^{n}S^2_{ijk}x_jx_k\Big), \quad i \in [n].\end{equation*}
The parameters are $r_i=A_{i,n+1,n+1}-A_{n+1,n+1,n+1}$, $S^1_{ij}=2(A_{ij,n+1}-A_{n+1,j,n+1})$, and $S^2_{ijk}=A_{ijk}-A_{n+1,jk}$. In the canonical representative, $A^0_{i,n+1,n+1}=r_i$, $A^0_{ij,n+1}=A^0_{i,n+1,j}=S^1_{ij}/2$, $A^0_{ijk}=S^2_{ijk}$, and the last row slice is zero. We remark that this model often appears in literature as higher-order Lotka--Volterra, although by the latter we mean any instance of \labelcref{eq: hoglv general}, where $d>2$. 
\end{example}

We now use the polynomial structure of \labelcref{eq: hoglv general} to bound the number of feasible equilibria, that is, the number of positive common zeros of $G_1,\dots,G_n$ from \eqref{eq: Gi polynomials}, each of degree at most $d-1$. By \cref{thm: higher hofbauer}, these bounds also apply to the internal equilibria of the associated replicator systems. The first bound is well-known in the evolutionary multiplayer games literature, see \cite{gokhale2010evolutionary,gokhale2014evolutionary, han2012equilibrium}. We recover it in the higher-order Lotka--Volterra setting directly from the Bernstein--Kushnirenko--Khovanskii (BKK) theorem, done similarly in e.g., \cite{AboEtAl2025VectorBundle, MCKELVEY1997411}, and then obtain a sharper bound under the squarefree-support assumption in Proposition~\ref{prop: squarefreesharp}.

\begin{proposition}
Assume that the system $G_1=\cdots=G_n=0$ has finitely many solutions in $(\mathbb C^\ast)^n$. Then the number of feasible equilibria of \labelcref{eq: hoglv general} is at most $(d-1)^n$. Moreover, for a generic system with full support, this bound is sharp for real feasible equilibria.
\end{proposition}

\begin{proof}
Let $k=d-1$. For every $i\in[n]$, let $P_i$ be the Newton polytope of $G_i$, namely the convex hull of the exponent vectors of the monomials that occur in $G_i$. Since every $G_i$ has degree at most $k$, we have $P_i\subseteq k\Delta_n$. By assumption, every common zero in $(\mathbb C^\ast)^n$ is isolated. The BKK theorem bounds the number of these zeros, counted with multiplicity, by the mixed volume $\operatorname{MV}(P_1,\dots,P_n)$, where we use the normalization $\operatorname{MV}(P,\dots,P)=n!\operatorname{vol}(P)$. By monotonicity of mixed volume, this is at most $\operatorname{MV}(k\Delta_n,\dots,k\Delta_n)=n!\operatorname{vol}(k\Delta_n)=k^n$. Every feasible equilibrium is one of these zeros, so the number of feasible equilibria is at most $k^n=(d-1)^n$. If the system has full support, then every monomial of degree at most $k$ occurs and $P_i=k\Delta_n$ for every $i$. Systems with full supports form a nonempty Zariski open subset of the coefficient space. For a general system in this open set all roots are simple, and BKK theorem gives exactly $\operatorname{MV}(k\Delta_n,\dots,k\Delta_n)=k^n$ solutions in $(\mathbb C^\ast)^n$.

Lastly, the bound $k^n = (d-1)^n$ is sharp also for real feasible equilibria. Indeed, we may choose the parameters so that $G_i(x)=\prod_{a=1}^{k}(x_i-a)$ for every $i\in[n]$. These polynomials have degree $k$ and hence define an instance of \labelcref{eq: hoglv general}. Their common zero set is $\{1,\dots,k\}^n$, so the system has exactly $k^n$ feasible equilibria. Moreover, every such zero is nonsingular. Indeed, at $\alpha\in\{1,\dots,k\}^n$ the Jacobian of $G=(G_1,\dots,G_n)$ is diagonal, and its $i$-th diagonal entry is $\prod_{a\in[k]\setminus\{\alpha_i\}}(\alpha_i-a)\neq0$. Since these zeros are nonsingular and lie in $\R^n_{>0}$, the implicit function theorem implies that, under a sufficiently small real perturbation of the coefficients, each of them deforms to a unique nearby simple zero which remains in $\R^n_{>0}$. Such a perturbation may be chosen so that every polynomial coefficient is non-zero. Thus, the bound is attained also on a nonempty Euclidean open set of full-support real systems.
\end{proof} \vspace{0.1cm}

\noindent So far, the partner indices in a higher-order interaction may repeat. When one wants to model only joint effects of distinct partner species, the natural restriction is the following.

\begin{definition}
We call the subsystem of \labelcref{eq: hoglv general} in which $S^q_{i j_1\cdots j_q}=0$, the \textit{squarefree higher-order Lotka--Volterra system}, whenever two of the indices $j_1,\dots,j_q$ are equal.
\end{definition}

In this case, each set $J$ of distinct species has at most one joint contribution to the growth rate of a focal species $i$. When present, the sign of this contribution is therefore unambiguous. Notice that the system in \cref{ex: higher obligate mutualism} is squarefree. The squarefree restriction makes the Newton polytope smaller, and it improves our equilibrium bound. 

\begin{proposition}\label{prop: squarefreesharp}
Let $k=d-1$ and assume $1\leq k<n$. If every $G_i$ has squarefree support contained in $\{\alpha\in\{0,1\}^n:|\alpha|\leq k\}$ and the common zero set in $(\mathbb C^\ast)^n$ is finite, then the number of feasible equilibria is at most
\begin{equation*}
    \sum_{\ell=0}^{k}(-1)^\ell\binom{n}{\ell}(k-\ell)^n.
\end{equation*}
In particular, this number is the normalized volume of the independence polytope of the uniform matroid $U_{k,n}$.
\end{proposition}

\begin{proof}
For every $i\in[n]$, let $P_i$ be the Newton polytope of $G_i$ and set $P=\operatorname{conv}\{\alpha\in\{0,1\}^n:|\alpha|\leq k\}$. The squarefree support assumption gives $P_i\subseteq P$ for every $i$. Moreover, $P$ is the independence polytope of the uniform matroid $U_{k,n}$, that is, the convex hull of the indicator vectors of its independent sets. Moreover, by Edmonds' description \cite{Edmonds1970} of independence polytopes, $P=\{u\in[0,1]^n \ | \ u_1+\cdots+u_n\leq k\}$. By the BKK theorem, and the monotonicity of mixed volume, the number of common zeros in $(\mathbb C^\ast)^n$, counted with multiplicity, is at most $\operatorname{MV}(P_1,\dots,P_n)\leq\operatorname{MV}(P,\dots,P)=n!\operatorname{vol}(P)$, which is the stated expression by e.g., \cite[Theorem 1]{MarichalMossinghoff2008} taking $w=(1,\cdots,1)$ and $z=k$.
\end{proof}\vspace{0.1cm}

\noindent Note also that this upper bound is the partial sum of Eulerian numbers which counts the number of permutations in $S_n$ with at most $k-1$ descents.

\subsection{A Game-{T}heoretic Perspective}\label{sec: 4.2}

We now interpret the fitness tensor $A$ in \labelcref{eq: higher fitness} as the payoff tensor of a symmetric $d$-player game (see \cite{gokhale2014evolutionary} for a comprehensive overview): its first index corresponds to the strategy of a focal player, while the remaining $d-1$ indices record the strategies of the other players. Under this interpretation, internal equilibria of the replicator dynamics, and hence feasible equilibria of the associated higher-order Lotka--Volterra system, correspond to totally mixed symmetric Nash equilibria. We begin by recalling the relevant notions for normal-form games.

Let $[d]=\{1,\dots,d\}$ be the set of players. Fix for each player $i$ a set of pure strategies $[m_i+1]$ and a mixed strategy simplex $\Delta_{m_i}$. A normal-form game is a $d$-tuple $X=(X^{(1)},\dots,X^{(d)})$, where $X^{(i)}\in\R^{(m_1+1)\times\cdots\times(m_d+1)}$ is the payoff tensor of player $i$, so that if the $d$ players play pure strategies $j_1,\dots, j_d$ respectively, then player $i$ receives $X^{(i)}_{j_1\cdots j_d}$.

A mixed strategy for player $i$ is a probability vector $p^{(i)}=(p^{(i)}_1,\dots,p^{(i)}_{m_i+1})\in\Delta_{m_i}$. The entries satisfy
\begin{equation*}
    \P(\text{player $i$ plays strategy $a$})=p^{(i)}_a,\quad \sum_{a=1}^{m_i+1}p^{(i)}_a=1.
\end{equation*}
A mixed strategy profile of the game is $p=(p^{(1)},\dots,p^{(d)})\in\Delta_{m_1}\times\cdots\times\Delta_{m_d}$. We assume that the players choose their mixed strategy independently, so that $p$ induces the product probability distribution
\begin{equation*}
    \P(\text{players }1,\dots,d\text{ play }j_1,\dots,j_d)=p^{(1)}_{j_1}\cdots p^{(d)}_{j_d}.
\end{equation*}
Let $p\in \Delta_{m_1}\times\cdots\times\Delta_{m_d}$. The expected payoff of player $i$ for $p$ is given by the multilinear expression
\begin{equation*}
    \Pi_i(p)\coloneq X^{(i)}\cdot p=\sum_{j_1=1}^{m_1+1}\cdots\sum_{j_d=1}^{m_d+1}X^{(i)}_{j_1\cdots j_d}p^{(1)}_{j_1}\cdots p^{(d)}_{j_d}.
\end{equation*}

Write $p_{-i}=(p^{(1)},\dots,p^{(i-1)},p^{(i+1)},\dots,p^{(d)})$. If player $i$ deviates to the pure strategy $a\in[m_i+1]$ while the other players keep strategies $p_{-i}$, the payoff of player $i$ is
\begin{equation*}
    \Pi_i(p_{-i},a)=\sum_{(j_h)_{h\neq i}\in\prod_{h\neq i}[m_h+1]}X^{(i)}_{j_1\cdots j_{i-1}a j_{i+1}\cdots j_d}\prod_{h\neq i}p^{(h)}_{j_h}.
\end{equation*}
For a mixed deviation $q\in\Delta_{m_i}$, write $\Pi_i(p_{-i},q)=\sum_{a=1}^{m_i+1}q_a\Pi_i(p_{-i},a)$. A mixed strategy profile $p$ is a Nash equilibrium if
\begin{equation}\label{eq: nash condition}
    \Pi_i(p)\geq \Pi_i(p_{-i},q)
\end{equation}
for every player $i$ and every $q\in\Delta_{m_i}$. It is a totally mixed Nash equilibrium if in addition to \labelcref{eq: nash condition}, we also have $p^{(i)}_a>0$ for every player $i$ and every $a\in[m_i+1]$. Standard computations (see \cite{AboEtAl2025VectorBundle}) show that the totally mixed Nash equilibria are described by
\begin{equation}\label{eq: mixed normal form}
\begin{cases}
    \Pi_i(p_{-i},1)=\cdots=\Pi_i(p_{-i},m_i+1) &\text{for each } i\in[d],\\
    p^{(i)}_1+\cdots+p^{(i)}_{m_i+1}=1 &\text{for each } i\in[d],\\
    p^{(i)}_a>0 &\text{for each } i\in[d]\text{ and }a\in[m_i+1].
\end{cases}
\end{equation}

\noindent Assume now that every player has the same pure strategy set $[n+1]$ and mixed strategy simplex $\Delta_n$. A game is symmetric if relabeling the players does not change payoffs, that is, for every permutation $\sigma\in\mathfrak S_d$, every pure profile $(j_1,\dots,j_d)$, and every player $i$,
\begin{equation}\label{eq: symmetric game}
    X^{(\sigma(i))}_{j_{\sigma^{-1}(1)}\cdots j_{\sigma^{-1}(d)}}=X^{(i)}_{j_1\cdots j_d}.
\end{equation}
The game can then be fully described by means of one focal player: once player $1$ is chosen, the other payoff tensors are recovered by relabeling the players. Since all players have the same mixed strategy simplex, every $y\in\Delta_n$ defines the diagonal mixed strategy profile $(y,\dots,y)\in(\Delta_n)^d$. Such a profile is called a symmetric mixed strategy profile. It is totally mixed if $y\in\operatorname{int}\Delta_n$. A symmetric mixed strategy profile $(y,\dots,y)$ is a symmetric Nash equilibrium if it is a Nash equilibrium of the $d$-player game. It is a totally mixed symmetric Nash equilibrium if, in addition, $y\in\operatorname{int}\Delta_n$.

Now, fix player $1$ as the focal player and define $A_{i k_1\cdots k_{d-1}}\coloneq X^{(1)}_{i k_1\cdots k_{d-1}}$. Let $\rho\in\mathfrak S_{d-1}$ be a permutation of the co-player positions, and let $\widehat\rho\in\mathfrak S_d$ be the permutation fixing player $1$ and satisfying $\widehat\rho(r+1)=\rho(r)+1$ for $r \in [d-1]$. Apply \labelcref{eq: symmetric game} with $\sigma=\widehat\rho$ and with the pure strategy profile $(i,k_1,\dots,k_{d-1})$. Since $\widehat\rho$ fixes player $1$, this gives $A_{i k_{\rho^{-1}(1)}\cdots k_{\rho^{-1}(d-1)}}=A_{i k_1\cdots k_{d-1}}$, and by the arbitrariness of $\rho$, we conclude that the focal payoff tensor of a symmetric normal-form game is symmetric in its last $d-1$ indices. Thus every symmetric $d$-player game determines a focal payoff tensor $A\in\operatorname{Hom}(\operatorname{Sym}^{d-1}\R^{n+1},\R^{n+1})$. Applying $\Psi_{d-1}$ to this tensor gives a higher-order Lotka--Volterra system.

For $y\in\Delta_n$ and $i\in[n+1]$, the payoff of pure strategy $i$ against $d-1$ co-players who all use $y$ is
\begin{equation}\label{eq: symmetric payoff}
    \Pi_i^A(y)=\sum_{k_1,\dots,k_{d-1}=1}^{n+1}A_{i k_1\cdots k_{d-1}}y_{k_1}\cdots y_{k_{d-1}}.
\end{equation}
These payoffs are precisely the polynomial fitnesses in \labelcref{eq: higher fitness}. We have the following result.

\begin{theorem}
Let $A\in\operatorname{Hom}(\operatorname{Sym}^{d-1}\R^{n+1},\R^{n+1})$ be the focal payoff tensor of a symmetric $d$-player game, and let $\Psi_{d-1}(A)=(r,S^1,\dots,S^{d-1})$. For $y\in\operatorname{int}\Delta_n$ and $x=\Phi(y)\in\R^n_{>0}$, the following statements are equivalent:
\begin{enumerate}
    \item \label{item: nash equilibrium} The profile $(y,\dots,y)$ is a totally mixed symmetric Nash equilibrium of the symmetric game.
    \item \label{item: replicator equilibrium} The state $y$ is an internal equilibrium of the replicator system \labelcref{eq: higher replicator}. 
    \item \label{item: lv equilibrium} The state $x$ is a feasible equilibrium of the higher-order Lotka--Volterra system \labelcref{eq: hoglv general}. 
\end{enumerate}
\end{theorem}

\begin{proof}
Since $y\in\operatorname{int}\Delta_n$, \labelcref{eq: mixed normal form} gives that $(y,\dots,y)$ is a totally mixed Nash equilibrium if and only if for every player, all pure strategies give the same payoff against the other $d-1$ players. By symmetry, it is enough to check this for the focal player, where the co-players all use $y$. As observed, these payoffs are the fitness functions in \labelcref{eq: higher fitness}, so \labelcref{item: nash equilibrium} is equivalent to $f_1(y)=\cdots=f_{n+1}(y)$, and so equivalent to \labelcref{item: replicator equilibrium}. The equivalence between Item \labelcref{item: replicator equilibrium} and Item \labelcref{item: lv equilibrium} in the theorem above is \cref{thm: higher hofbauer}.
\end{proof}

\begin{remark}
If one starts with an arbitrary focal payoff tensor $A\in\R^{(n+1)^d}$, not necessarily symmetric in the co-player indices, then replacing $A$ by its co-player symmetrization
\[
\operatorname{Sym}_{\mathrm{co}}(A)_{i k_1\cdots k_{d-1}}\coloneq \frac{1}{(d-1)!}\sum_{\sigma\in\mathfrak S_{d-1}}A_{i k_{\sigma(1)}\cdots k_{\sigma(d-1)}}
\]
does not change the totally mixed symmetric Nash equilibria of the game. This operation does not preserve non-symmetric Nash equilibria in general.
\end{remark}

We now ask whether a totally mixed symmetric Nash equilibrium associated to a feasible and stable equilibrium persists when the payoffs are perturbed. We first recall the notion of pay-off robustness.

\begin{definition}
A Nash equilibrium $p_\star$ of a normal-form game $X^0$ is called \emph{payoff-robust} if, for every neighborhood $V$ of $p_\star$, there exists $\varepsilon>0$ such that every normal-form game $\widetilde X$ with the same strategy sets and satisfying $\|\widetilde X-X^0\|<\varepsilon$ admits a Nash equilibrium $\widetilde p\in V$.
\end{definition}

Let $A^0$ be the canonical representative from \cref{def: canonical representative}, let $X^0$ be the symmetric normal-form game whose focal payoff tensor is $A^0$, let $x_\star$ be a feasible equilibrium of \labelcref{eq: hoglv general}, and set $y_\star=\Phi^{-1}(x_\star)$. Write $\tilde y_\star\coloneq(y_\star,\dots,y_\star)$.

\begin{theorem}\label{thm: payoff robust}
If the Jacobian $J(x_\star)$ of $\dot x=\diag(x)G(x)$ at $x_\star$ is nonsingular, then $\tilde y_\star$ is payoff-robust. More precisely, every sufficiently small perturbation $\widetilde X$ of $X^0$, not necessarily symmetric, admits a totally mixed Nash equilibrium $\widetilde p$ near $\tilde y_\star$. If $\widetilde X$ is symmetric, the nearby equilibrium can be chosen symmetric. In particular, these conclusions hold when $x_\star$ is linearly asymptotically stable.
\end{theorem}

\begin{proof} 
For $u=(u_1,\dots,u_n)\in\R^n$ with $u_a>0$ for every $a\in[n]$ and $\sum_{a=1}^n u_a<1$, we define $q(u)\coloneq(u_1,\dots,u_n,1-\sum_{a=1}^n u_a)\in\operatorname{int}\Delta_n$. For each player $h$, we write $p^{(h)}=q(z^{(h)})$ and set $z=(z^{(1)},\dots,z^{(d)})$. For a normal-form game $X$, we define 
\[
N^X_{h,a}(z)\coloneq\Pi_h^X(p_{-h},a)-\Pi_h^X(p_{-h},n+1) \quad \text{for } h\in[d] \text{ and } a\in[n]. 
\] 
We set $N_h^X\coloneq (N^X_{h,1},\dots,N^X_{h,n})$ and $N^X\coloneq (N_1^X,\dots,N_d^X)$. Since every $p^{(h)}$ is totally mixed, the equations $N^X(z)=0$ are equivalent to the totally mixed Nash equilibrium conditions in \labelcref{eq: mixed normal form}. 

We call $N^X(z)=0$ the full Nash system, since the strategies of the $d$ players vary independently. For every focal payoff tensor $A$, we define 
\[
F_a^A(u)\coloneq\Pi_a^A(q(u))-\Pi_{n+1}^A(q(u)) \quad \text{for } a\in[n], 
\]
and $F^A\coloneq(F_1^A,\dots,F_n^A)$. If $X$ is symmetric with focal payoff tensor $A$, symmetry gives $N_h^X(u,\dots,u)=F^A(u)$ for every $h\in[d]$. Therefore, $F^A(u)=0$ if and only if $(q(u),\dots,q(u))$ is a totally mixed symmetric Nash equilibrium of $X$. We call $F^A(u)=0$ the symmetric Nash system. 

Let $u_\star\coloneq(y_{\star,1},\dots,y_{\star,n})$ and $\tilde z_\star\coloneq(u_\star,\dots,u_\star)$, so that $\tilde z_\star$ is the coordinate representation of $\tilde y_\star$. We regard $D_zN^{X^0}(\tilde z_\star)$ as a $d\times d$ block matrix whose $(h,k)$-block is $D_{z^{(k)}}N_h^{X^0}(\tilde z_\star)$. The diagonal blocks vanish because the payoff of player $h$ from a pure deviation depends only on the strategies of the other players. Since $X^0$ is symmetric, $A^0$ is symmetric in the co-player indices, and $\tilde y_\star$ is diagonal, all the off-diagonal blocks are equal. We denote this common $n\times n$ block by $B$. Thus, 
\[D_zN^{X^0}(\tilde z_\star)=(E_d-I_d)\otimes B,\] where $E_d$ is the $d\times d$ matrix of all ones and $\otimes$ denotes the Kronecker product. Since $F^{A^0}(u)=N_1^{X^0}(u,\dots,u)$, the chain rule gives 
\[D_uF^{A^0}(u_\star)=\sum_{k=1}^dD_{z^{(k)}}N_1^{X^0}(\tilde z_\star)=(d-1)B.\]
For the canonical representative $A^0$, the computation in the proof of \cref{thm: higher hofbauer} gives 
\[F^{A^0}(u)=q_{n+1}(u)^{d-1}G(\Phi(q(u))).\] 
Since $\Phi(q(u_\star))=x_\star$ and $G(x_\star)=0$, differentiating at $u_\star$ gives \begin{equation}\label{eq: payoff jacobian} D_uF^{A^0}(u_\star)=y_{\star,n+1}^{d-1}D_xG(x_\star)D_u(\Phi\circ q)(u_\star). \end{equation} The map $\Phi\circ q$ is a diffeomorphism onto $\R^n_{>0}$, so $D_u(\Phi\circ q)(u_\star)$ is invertible. Moreover, 
\[J(x_\star)=\diag(G(x_\star))+\diag(x_\star)D_xG(x_\star)=\diag(x_\star)D_xG(x_\star).\] 
Since $\diag(x_\star)$ is invertible, the nonsingularity of $J(x_\star)$ implies that $D_xG(x_\star)$ is nonsingular. By \labelcref{eq: payoff jacobian}, $D_uF^{A^0}(u_\star)$ is nonsingular, and hence $B$ is nonsingular. The matrix $E_d-I_d$ has eigenvalue $d-1$ on the span of the all-ones vector and eigenvalue $-1$ with multiplicity $d-1$. Therefore \begin{equation}\label{eq: full nash jacobian} \det D_zN^{X^0}(\tilde z_\star)=(-1)^{n(d-1)}(d-1)^n\det(B)^d\neq0. \end{equation} The implicit function theorem applied to the map $(z,X)\mapsto N^X(z)$ gives a nearby zero $\widetilde z$ for every normal-form game $\widetilde X$ sufficiently close to $X^0$. Given a neighborhood $V$ of $\tilde y_\star$, we may shrink the neighborhoods in the implicit function theorem so that the mixed strategy profile $\widetilde p$ corresponding to $\widetilde z$ lies in $V$ and remains totally mixed. Since $N^{\widetilde X}(\widetilde z)=0$, the profile $\widetilde p$ is a totally mixed Nash equilibrium of $\widetilde X$. 

Suppose now that $\widetilde X$ is symmetric, and let $\widetilde A$ be its focal payoff tensor. The implicit function theorem applied to the map $(u,A)\mapsto F^A(u)$ gives a nearby zero $\widetilde u$ of $F^{\widetilde A}$ because $D_uF^{A^0}(u_\star)$ is nonsingular. After shrinking the neighborhoods, $\widetilde y\coloneq q(\widetilde u)$ lies in $\operatorname{int}\Delta_n$ and $(\widetilde y,\dots,\widetilde y)$ lies in $V$. Since $F^{\widetilde A}(\widetilde u)=0$, this profile is a totally mixed symmetric Nash equilibrium of $\widetilde X$. 

Finally, if $x_\star$ is linearly asymptotically stable, then every eigenvalue of $J(x_\star)$ has negative real part. Therefore, $J(x_\star)$ is nonsingular, so the above conclusions hold. 
\end{proof} 

\begin{remark}
{The nonsingularity proved above connects \cref{thm: payoff robust} with the classical index theory of Nash equilibria; see \cite{kohlberg1986strategic,ritzberger1994theory,demichelisgermano2000indices}. In our notation, the Nash equilibrium index of $\tilde y_\star$ is the local degree of $-N^{X^0}$ at $\tilde z_\star$. Since $D_uF^{A^0}(u_\star)=(d-1)B$, $\det D_u(\Phi\circ q)(u_\star)=y_{\star,n+1}^{-(n+1)}>0$, and $\det\diag(x_\star)>0$, we have $\sign\det B=\sign\det J(x_\star)$. Hence, \eqref{eq: full nash jacobian} gives
\[
\operatorname{ind}(\tilde y_\star)=\operatorname{sign}\det\bigl(-D_zN^{X^0}(\tilde z_\star)\bigr)=(-1)^n\bigl(\operatorname{sign}\det J(x_\star)\bigr)^d.
\]
In particular, this index is nonzero. The persistence property of nonzero equilibrium index therefore gives another proof of payoff robustness under arbitrary sufficiently small payoff perturbations \cite{ritzberger1994theory}. Similarly, we have that the symmetric totally mixed equilibrium system $F^{A^0}(u)=0$ has local degree $\operatorname{sign}\det J(x_\star)$ at $u_\star$. Its nonzero local degree guarantees persistence of a nearby totally mixed symmetric Nash equilibrium under sufficiently small symmetric payoff perturbations.

Index theory plays an important role in the study of payoff robustness, and recent algebraic approaches offer new tools for computing equilibrium indices, including in cases with singular Jacobian; see, for instance, \cite{pahl2026index}.} Moreover, for the cohomological foundations of stability, see Mertens \cite{mertens1989stable} and Govindan and Mertens \cite{govindan2004equivalent}.
\end{remark}

\begin{corollary} 
For a classical Lotka--Volterra system $\dot x=\diag(x)(r+Sx)$, every feasible equilibrium $x_\star=-S^{-1}r\in\R^n_{>0}$ gives a payoff-robust totally mixed symmetric Nash equilibrium of the associated symmetric two-player game. 
\end{corollary} 

\begin{proof} 
In this case $J(x_\star)=\diag(x_\star)S$. Since $x_\star\in\R^n_{>0}$ and $S$ is invertible, $J(x_\star)$ is nonsingular. The claim follows from \cref{thm: payoff robust}. 
\end{proof}
\begin{corollary}\label{cor: adding players} 
For every $r\in\R^n$ and $S^1\in\R^{n^2}$, there are tensors $S^2\in\R^{n^3}$ and $S^3\in\R^{n^4}$ such that the canonical symmetric four-player game associated with $(r,S^1,S^2,S^3)$ admits the uniform diagonal profile $(y_\star,\dots,y_\star)$, where $y_\star=(1,\dots,1)/(n+1)$, as a payoff-robust totally mixed symmetric Nash equilibrium. 
\end{corollary} 

\begin{proof} Choose $S^2$ and $S^3$ as in \cref{thm: higher stabilization} and symmetrize their partner indices, which leaves the vector field unchanged. By \cref{prop: canonical representative}, these parameters determine the canonical focal payoff tensor of a symmetric four-player game. The associated higher-order Lotka--Volterra system has the feasible equilibrium $x_\star=\mathbf 1$ with nonsingular Jacobian and $\Phi^{-1}(\mathbf 1)=y_\star$. The claim follows from \cref{thm: payoff robust}. 
\end{proof}

\subsection{Impossible Ecologies and Evolutionarily Stable Strategies}\label{subsec: ess} 

We are ready to give the evolutionary meaning of the impossible ecologies. A Nash equilibrium says that no individual benefits from deviating when the rest of the population keeps its strategy, but it does not say whether the population can resist a small group of mutants. This is the role of an \textit{evolutionarily stable strategy}, introduced by Maynard Smith and Price (\cite{MaynardSmithPrice1973,MaynardSmith1974}) to describe a resident strategy that cannot be invaded by a sufficiently rare mutant. We retain the classical notion for symmetric two-player games.

Let $A\in\R^{(n+1)\times(n+1)}$ be the focal payoff matrix of a symmetric two-player game, and write $\pi(p,q)\coloneq p^{\mathsf T}Aq$ for the payoff of $p\in\Delta_n$ against $q\in\Delta_n$. Here, we refer to $p$ as the resident strategy and to $q$ as the mutant strategy.

\begin{definition}
A strategy $p\in\Delta_n$ is an \emph{evolutionarily stable strategy} (ESS) if, for every $q\in\Delta_n$ with $q\neq p$, there is an $\varepsilon_q>0$ such that $\pi(p,(1-\varepsilon)p+\varepsilon q)>\pi(q,(1-\varepsilon)p+\varepsilon q)$ for every $0<\varepsilon<\varepsilon_q$. It is \emph{totally mixed} if $p\in\operatorname{int}\Delta_n$.
\end{definition}

For two-player games, this definition has an equivalent form: for every $q\neq p$, either $\pi(p,p)>\pi(q,p)$, or $\pi(p,p)=\pi(q,p)$ and $\pi(p,q)>\pi(q,q)$ (see \cite{MaynardSmith1974,HofbauerSchusterSigmund1979}). Thus an ESS is a symmetric Nash equilibrium, with an additional test when the mutant is itself a best reply to the resident. In that case, residents must earn more against mutants than mutants earn against one another. For a totally mixed symmetric Nash equilibrium $p$, every mutant $q$ ties with $p$ against the resident population, so the ESS condition reduces to $\pi(p,q)>\pi(q,q)$ for every $q\neq p$.

ESS have a crucial link with dynamics. Taylor and Jonker proved the stability result for regular ESSs, and Hofbauer, Schuster, and Sigmund showed that every ESS is a locally asymptotically stable equilibrium of the replicator dynamics (see \cite{TaylorJonker1978,HofbauerSchusterSigmund1979,Hofbauer1998} or \cite{Allesina2026} for a recent treatment through the equivalence between replicator and Lotka--Volterra dynamics). Therefore, by \cref{thm: higher hofbauer} with $d=2$, a totally mixed ESS gives a feasible locally asymptotically stable equilibrium of the associated Lotka--Volterra system.

We now apply these results to the sign patterns studied in \cref{sec:impossible-ecologies}. Fix an ecological sign pattern $\sigma=(\sigma_i,\sigma_{ij})\in\{\pm\}^{n+n^2}$. Given a symmetric two-player game with focal payoff matrix $A\in\R^{(n+1)\times(n+1)}$, let $A^0$ denote the canonical representative of the class of $A$ modulo $\ker\Psi_1$, as in \cref{def: canonical representative}. We set $a_i(A)\coloneq A^0_{i,n+1}=A_{i,n+1}-A_{n+1,n+1}$ and $b_{ij}(A)\coloneq-A^0_{ij}=A_{n+1,j}-A_{ij}$, and write $\ba(A)\coloneq(a_i(A))_{i\in[n]}$ and $\bB(A)\coloneq(b_{ij}(A))_{i,j\in[n]}$. Thus, $A^0$ has zero last row, its last column above the final entry is $\ba(A)$, and its upper-left block is $-\bB(A)$. Let $\mathcal G_\sigma$ be the class of symmetric two-player games for which $\sign(\ba(A),\bB(A))=\sigma$, that is, the class of games whose replicator dynamics correspond to Lotka--Volterra systems $\dot{\bx}=\diag(\bx)(\ba(A)-\bB(A)\bx)$ with sign pattern $\sigma$. 

\begin{theorem}\label{thm: impossible ess} 
If $\sigma$ is an impossible ecological sign pattern on $n$ species, then no game in $\mathcal G_\sigma$ admits a totally mixed evolutionarily stable strategy.
\end{theorem} 

\begin{proof} Suppose that a game in $\mathcal G_\sigma$ has a totally mixed ESS $p$. The stability discussion above makes $p$ a locally asymptotically stable internal equilibrium of its replicator dynamics. Under the diffeomorphism $\Phi$ of \labelcref{eq: diffeo simplex}, the point $p$ maps to a feasible locally asymptotically stable equilibrium ${\xstar}=\Phi(p)$ of the associated Lotka--Volterra system by \cref{thm: higher hofbauer}. This system has sign pattern $\sigma$, contradicting the assumption that $\sigma$ is impossible. 
\end{proof} \vspace{0.1cm}

We highlight that the statement concerns whole classes of games, not isolated payoff matrices. In fact, by \cref{prop: canonical representative}, the map $A\mapsto(\ba(A),\bB(A))$ is surjective. Moreover, the signs defining $\mathcal G_\sigma$ are given by strict linear inequalities in the entries of $A$. Hence, every $\mathcal G_\sigma$ is a nonempty open polyhedral set of payoff matrices. Therefore, the obstruction comes from the signs of payoff differences encoded by the ecological network, rather than from the format of the game alone. 

\begin{corollary}\label{cor: formats without ess} 
There are nonempty open classes of symmetric two-player games of formats $2\times2$, $3\times3$, and $4\times4$ that admit no totally mixed ESS. The infinite families of impossible ecologies in \cref{sec:impossible-ecologies} give such classes of $(n+1)\times(n+1)$ games for arbitrarily large $n$. 
\end{corollary} 

\begin{proof} 
For one species, the sign pattern $a_1<0$ and $b_{11}>0$ is trivially impossible. The impossible two and three species ecologies are recalled in \cref{sec:FSstratum}, while the four species and the infinite families are proved in \cref{sec:impossible-ecologies}. Applying \cref{thm: impossible ess} gives the claim. 
\end{proof} 

\begin{example}\label{ex: no_ess_5x5}
Consider a symmetric two-player game where each player has $5$ pure strategies, corresponding to an ecological system with $n=4$ species. Suppose the game is defined by the following payoff matrix $X \in \R^{5 \times 5}$:
\[
X = \begin{pmatrix}
 2 & 5 & 4 & 3 & 6 \\
 4 & 2 & 6 & 5 & 3 \\
 2 & 5 & 4 & 8 & 1 \\
 2 & 5 & 9 & 3 & 1 \\
 3 & 6 & 5 & 4 & 2
\end{pmatrix}.
\]
We ask whether this game admits a totally mixed ESS. We can construct its canonical representative defined in \cref{def: canonical representative} to uncover its underlying ecological network. Following the notation of \cref{subsec: ess}, we compute the growth rates $\ba(X) \in \R^4$ and the interaction matrix $\bB(X) \in \R^{4 \times 4}$. The components of the growth vector are given by $a_i(X) = X_{i,5} - X_{5,5}$, while the interaction coefficients are given by $b_{ij}(X) = X_{5,j} - X_{i,j}$:
\[
\ba(X) = \begin{pmatrix} 4 \\ 1 \\ -1 \\ -1 \end{pmatrix},\quad
\bB(X) = \begin{pmatrix} 1 & 1 & 1 & 1 \\ -1 & 4 & -1 & -1 \\ 1 & 1 & 1 & -4 \\ 1 & 1 & -4 & 1 \end{pmatrix}.
\]
Therefore, this game belongs to the class $\mathcal G_\sigma$, where $\sigma$ is the sign pattern of the asymmetric four species ecology of \cref{lem: eco1asymm}. The same lemma shows that this ecological network is impossible, meaning it cannot admit a feasible, locally asymptotically stable equilibrium. Therefore, by \cref{thm: impossible ess}, we immediately conclude that the symmetric game defined by $X$ cannot admit any totally mixed ESS. In fact, the uniform strategy $p_\star=(1/5,1/5,1/5,1/5,1/5)$ is the unique totally mixed symmetric Nash equilibrium of the game, but it is not an ESS.
\end{example}

\printbibliography[heading=bibintoc,title={References}]

\end{document}